\documentclass[pdflatex,sn-standardnature]{sn-jnl}
\usepackage{graphicx}%
\usepackage{multirow}%
\usepackage{amsmath,amssymb,amsfonts}%
\usepackage{amsthm}%
\usepackage{mathrsfs}%
\usepackage[title]{appendix}%
\usepackage{xcolor}%
\usepackage{textcomp}%
\usepackage{manyfoot}%
\usepackage{booktabs}%
\usepackage{enumitem}
\usepackage{mathtools}
\usepackage{algorithm}%
\usepackage{algorithmicx}%
\usepackage{algpseudocode}%
\usepackage{listings}%
\usepackage{subcaption}
\usepackage{hyperref}

\newcommand{\ccb}{\color{black}}

\newcommand{\ccn}{\color{black}}

\newcommand{\F}{\mathbb F}

\newcommand{\mQ}{\mathsf{Q}}
\newcommand{\ba}{\mathbf{a}}
\newcommand{\bb}{\mathbf{b}}
\newcommand{\bPhi}{\mathbf{\Phi}}
\newcommand{\bB}{\mathbf{B}}

\renewcommand{\sim}{\bowtie}

\definecolor{blue-violet}{rgb}{0.54, 0.17, 0.89}
\newcommand{\cli}{\color{black}}
\newcommand{\cl}{\color{black}}
\newcommand{\cn}{\color{black}}
\renewcommand{\bar}{\hat}

\theoremstyle{thmstyleone}%
\newtheorem{theorem}{Theorem}[section]%  meant for continuous numbers
\theoremstyle{theorem}%
\newtheorem{remark}[theorem]{Remark}%
\newtheorem{lemma}[theorem]{Lemma}
\newtheorem{corollary}[theorem]{Corollary}

\newtheorem{prop}[theorem]{Proposition}
\newtheorem{defn}[theorem]{Definition}
\newtheorem{prob}{Main Problem}

\usepackage{marginnote}

\newcounter{sidenote}
\reversemarginpar

\begin{document}
	
	\title[On the Codebook Design for NOMA Schemes from Bent Functions]{\bf On the Codebook Design for NOMA Schemes from Bent Functions
%		\footnote{This is an expanded and vastly improved version of the 5-page Extended Abstract presented at the 9th International Workshop on Boolean Functions and their Applications (BFA) in Dubrovnik, Croatia, in September 2024. The work of the fourth-named author (P.S.) was partially sponsored by DoD.}
	}
	% \author{
	% Chunlei Li$^1$, Constanza Riera$^2$, Pantelimon St\u anic\u a$^3$, Palash Sarkar$^1$
	% \vspace{.2cm}\\
	% \small $^1$ Department of Informatics, University of Bergen,  \\
	% \small  5020 Bergen, Norway; \tt{\{chunlei.li,palash.sarkar\}@uib.no}\\
	% \small  $^2$ Department of Computer Science, Electrical Engineering
	% and Mathematical Sciences,\\
	% \small Western Norway University of Applied Sciences,
	% 5020 Bergen, Norway; \tt {csr@hvl.no}\\
	% \small  $^3$  Applied Mathematics Department, Naval Postgraduate School,\\
	% \small Monterey, CA 93943, USA;  \tt{pstanica@nps.edu}
	
	\author[1]{\fnm{Chunlei} \sur{Li} \href{https://orcid.org/0000-0002-3792-769X}{\includegraphics[scale=0.01]{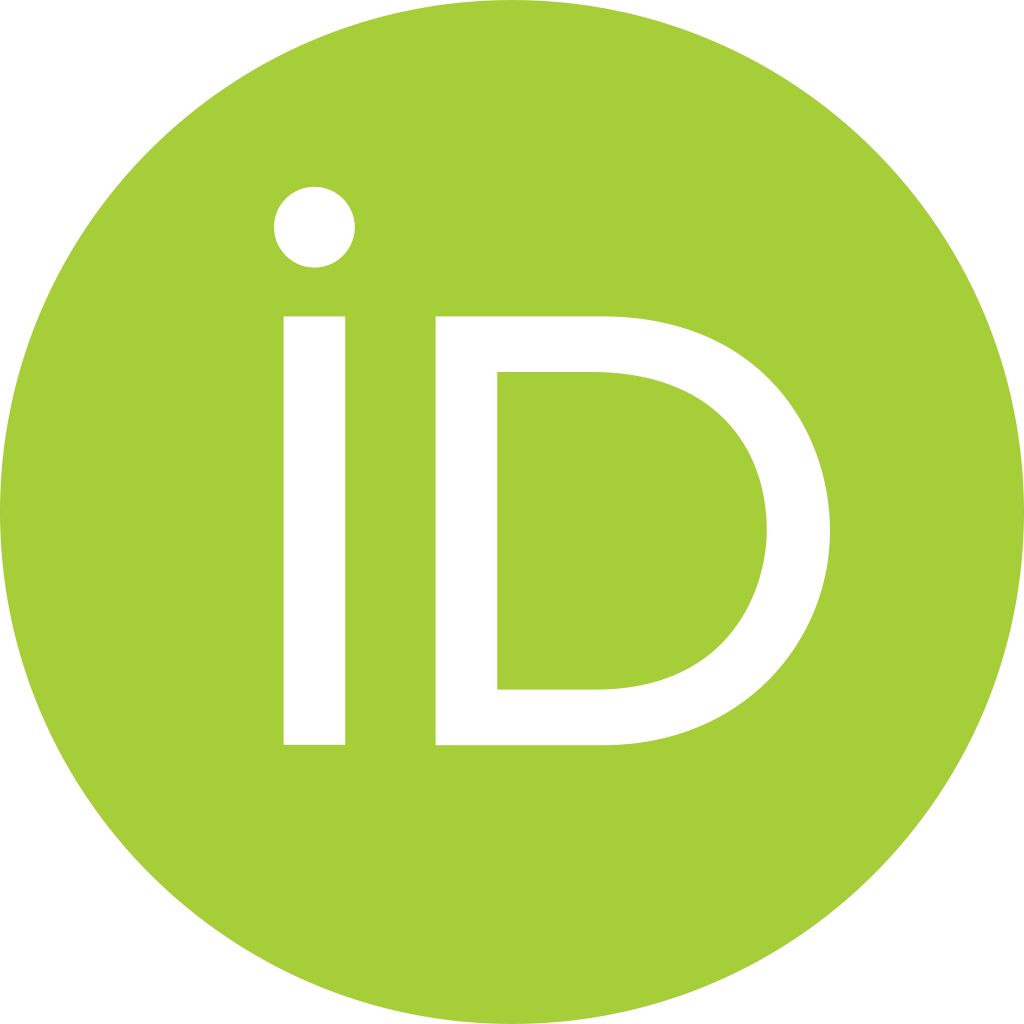}}
	} \email{\href{mailto:chunlei.li@uib.no}{chunlei.li@uib.no}}
	\author[2]{\fnm{Constanza} \sur{Riera} 
		\href{https://orcid.org/0000-0001-8829-0847}
		{\includegraphics[scale=0.01]{orcid_logo.png}}
	}
	\email{\href{mailto: csr@hvl.no}{csr@hvl.no}}
	
	\author[1]{\fnm{Palash} \sur{Sarkar} 
    % \href{https://orcid.org}{\includegraphics[scale=0.01]{orcid_logo.png}} insert correct orcid
    }
	\email{\href{mailto: palash.sarkar@uib.no}{palash.sarkar@uib.no}}
	
	\author[3]{\fnm{Pantelimon} \sur{St\u anic\u a} 
		\href{https://orcid.org/0000-0002-8622-7120}
		{\includegraphics[scale=0.01]{orcid_logo.png}}
	} 
	\email{\href{mailto: pstanica@nps.edu}{pstanica@nps.edu}}

	%
	%\author[1]{\fnm{Prashant Kumar} \sur{Srivastava} \href{https://orcid.org/0000-0002-7651-5639}{\includegraphics[scale=0.01]{orcid_logo.png}}}\email{pksri@iitp.ac.in}
	%\author[3]{\fnm{Sudhan} \sur{Majhi} \href{https://orcid.org/0000-0002-2142-1862}{\includegraphics[scale=0.01]{orcid_logo.png}}}\email{smajhi@iisc.ac.in}
	%%\equalcont{These authors contributed equally to this work.}
	%
	\affil[1]{
		\orgdiv{Department of Informatics}, 
		\orgname{University of Bergen}, 
		\orgaddress{Bergen, \country{Norway}}}
	
	\affil[2]{
		\orgdiv{Department of Computer Science, Electrical Engineering
			and Mathematical Sciences},
		\orgname{Western Norway University of Applied Sciences}, 
		\orgaddress{Bergen, Norway}}
	
	\affil[3]{
		\orgdiv{Applied Mathematics Department}, 
		\orgname{Naval Postgraduate School},
		\orgaddress{Monterey, CA 93943, USA}}
	
	
	\abstract{ 
		Uplink grant-free non-orthogonal multiple access (NOMA) is a promising technology for massive connectivity with low latency and high energy efficiency. In code-domain NOMA schemes, \cl the requirements boil down to the design of codebooks that contain a large number of spreading sequences with low peak-to-average power ratio (PAPR) while maintaining low coherence. When employing binary Golay sequences with guaranteed low PAPR in the design, \cn the fundamental problem is to construct a large set of $n$-variable quadratic bent or near-bent functions in a particular form \cli such that \cn the difference of any two is bent for even $n$ or near-bent for odd $n$ to achieve optimally low coherence. \cn  
		% Because of the low PAPR requirement, the Kerdock codes or vectorial bent functions cannot be applied in the design. 
		In this work, we propose a theoretical construction of NOMA codebooks by applying a recursive approach to those particular quadratic bent functions in smaller dimensions. The proposed construction yields desired NOMA codebooks \cli that \cn contain $6\cdot N$ Golay sequences of length $N=2^{4m}$ for any positive integer $m$ and have the lowest possible coherence $1/\sqrt{N}$.
		% Based on our proposed construction any maximal set in $4$ dimensions can be extended to a set in any dimension $4k$ while maintaining the same overloading factor of 6 as the original $4$-dimensional set. 
		% Here, the overloading factor of a set represents the ratio between the size of the set and the length of the sequences.
		
	}

	
	
	
	\maketitle
	
	\section{Introduction}
    \label{sec:intro}
	
	The massive connectivity of wireless devices is a fundamental aspect of machine-type communications \cite{mtc2012}, which  
	must accommodate a large number of devices while ensuring minimal control overhead, low latency, and low power consumption for delay-sensitive and energy-efficient communication. Non-orthogonal multiple access (NOMA) \cite{noma5g2015,surveynoma18} has \cli emerged \cn as a solution for enabling massive device connectivity in 5G wireless systems. 
	By permitting multiple devices to share common resources without scheduling, uplink grant-free NOMA is a promising approach to \cli achieving \cn massive connectivity with low latency and high energy efficiency~\cite{yulnomacodebook}. 
	In grant-free code-domain NOMA, user-specific spreading sequences are assigned to all devices, allowing active devices to minimize signaling overhead. \cl In this NOMA scheme, the codebook in use should consist of a large number of spreading sequences with low peak-to-average power ratio (PAPR) to enable massive connectivity and energy efficiency and should have low coherence to minimize interference, which is critical for the system performance. \ccn
	
	\cl
	The design of large sets of spreading sequences with low correlation has been extensively researched in the literature \cite{cding2006,cding2007,luo2018}. \cn However, \cli constructing \cn large sequence sets that not only exhibit low correlation but also maintain low PAPR (see \cli Definition~\ref{def:papr}) and a small alphabet size is challenging. \cn
	\cl 
	Golay complementary sequences~\cite{golay1949} are a pair or a set of sequences that have zero aperiodic correlation sum \cli at all nonzero shifts. Binary Golay sequences have PAPR upper bounded by 2 \cn(see Subsection~\ref{subsec:CS}). The utility of Golay sequences in controlling the PAPR of orthogonal frequency division multiplexing (OFDM) schemes has been recognized since the 1950s. In 1999, Davis and Jedwab~\cite{Davis1999} established a fundamental connection between Golay sequences and Reed-Muller (RM) codes:
	all existing binary Golay sequences of length $2^n$ correspond to cosets of the first-order RM codes in the second-order binary RM codes, where the coset representatives are $n$-variable quadratic Boolean functions (the reader can find basics on these objects in Subsection~\ref{subsec:BF}) of algebraic normal form $Q_{\pi}(x)=\sum_{i=1}^{n-1} x_{\pi(i)} x_{\pi(i+1)}$, where $x=(x_1,\ldots,x_n) \in \F_2^n$ and $\pi$ is a permutation on the set $\{1,2,\ldots,n\}$. \cn
	% More precisely, the Golay sequences of length $2^n$ can be 
	% defined as the truth table of the Boolean function 
	% $Q_{\pi}(x) + \sum_{i=1}^nc_ix_i$ for any permutation $\pi$ of $\{1,2,\hdots,n\}$ and any coefficients $c_i\in \{0,1\}$.
	Such binary sequences were later referred to as Golay-Davis-Jedwab (GDJ) sequences in the literature.
	This connection enabled Davis and Jedwab to construct a large codebook of Golay sequences to control power consumption in \cli OFDM \cn schemes~\cite{Davis1999}. 
	
	In code-domain NOMA, Yu \cite{yulnomacodebook}
	proposed \cli a codebook design for uplink grant-free NOMA schemes,
    where the codebook is a spreading matrix consisting of  
    GDJ sequences as its columns (see Definition~\ref{Def_Codebook_Design}).
	This ensures that each column in the spreading matrix has PAPR upper bounded by $2$.
	In addition, the spreading matrix \cn is desired to have low column-wise coherence (Definition~\ref{def_coherence}) for better performance in joint channel estimation and multiuser detection, \cli and to have a large number of columns to support massive connectivity.  
    These requirements pose a significant challenge in the codebook design.
	% Although some attempts were made on this topic~\cite{yulnomacodebook,tina2022,liu2023,vik2024}, this area remains largely unexplored.  
    \cl
	In Yu's proposed construction~\cite{yulnomacodebook}, an optimally low coherence of the codebook can be derived by choosing a set of $L$ permutations $\{\pi_1,\ldots,\pi_L\}$ on $\{1,2,\ldots,n\}$ such that for any two distinct permutations $\pi_{l_1}, \pi_{l_2}$,
	the quadratic function $Q_{\pi_{l_1}}(x)+Q_{\pi_{l_2}}(x)$, $x\in\F_2^n$, is bent for even $n$ and near-bent for odd $n$ (see Definition~\ref{walsh_bent}). 
	For the simplicity of presentation, two permutations $\pi,~\rho$ on the set $\{1,2,\dots, n\}$ will be referred to as \textit{compatible} in this paper, denoted as $\pi \sim \rho$, if the corresponding quadratic function $Q_{\pi}(x)+Q_{\rho}(x)$ is bent (resp., near-bent) when $n$ is even (resp., odd). 
	Based on Yu's construction, from $L$ mutually compatible permutations, one can easily derive a NOMA codebook of $L\cdot 2^n$ spreading sequences with PAPR upper bounded by $2$ and optimally low coherence. \cli The parameter $L$ is termed the \textit{overloading factor} of the spreading matrix in~\cite{yulnomacodebook}. 
	The key task is then to construct a set of $L$ mutually compatible permutations of $\{1,2,\dots, n\}$  with $L$ as large as possible. \cl 
	Yu obtained compatible sets through a computer search. 
	However, this method quickly becomes infeasible when $n$ is larger than or equal to $9$. Tian, Liu and Li in~\cite{tina2022} proposed a graph-based approach to construct a compatible set of size $L=4$. Very recently,  in~\cite{vik2024}, the authors used the quadratic Gold functions to generate a compatible set of maximum size $\frac{p-1}{2}$, where $p$ is the minimum prime factor of $n+1$ for even $n$ (resp. $n$ for odd $n$). 
	This approach can yield a relatively large compatible set for carefully chosen $n$, especially when $n+1$ or $n$ is an odd prime. On the other hand, the set size drops dramatically when the minimum prime factor of $n+1$ for even $n$ (resp. $n$ for odd $n$) is small. 
	% When $n\geq 4$, we observed through computer search that the maximal set size of compatible permutations, which we denote by $L$, is greater than or equal to~$6$. However, through the construction of~\cite{vik2024}, the permutation set size drops quickly when the \ccb minimum \ccn prime factor $p\leq 5$. 
	For example, for positive integers $n$ \cli such that
    $n$ modulo $6$ equals $2$ or $3$, the minimum factor of $n+1$ or $n$ is 3, and then $L = 1$; similarly, when $n$ modulo $10$ equals $4$ or $5$, the corresponding loading factor $L$ is only $2$. 
    \cl
    % from $16$ up to $27$, the constructed compatible sets in \cite{vik2024} have fluctuating sizes $L=8, 8, 9, 9, 1, 1, 11, 11, 2, 2, 1, 1$, respectively.   
    This motivates us to undertake an alternative approach to constructing codebooks for uplink grant-free NOMA schemes.

	In this paper, we propose a recursive construction of \cli a set of mutually compatible permutations of $\{1,2,\dots, n\}$, termed a \textit{compatible set}, for infinitely many $n$. \cli We denote by $S_n$ the set of all permutations of $\{1,2,\dots, n\}$, and by $IS_n$ the set of all permutations \cl
	that are compatible with the identity permutation $I_n$ of $\{1,2,\dots, n\}$.
	% Observe that any compatible set can be transformed into a compatible set containing $I_n$.
	We start by investigating the conditions under which \cli a permutation in $IS_{n+m}$ can be derived from \cl 
    a permutation $\pi$ in $IS_n$ and a permutation $\rho$ in $IS_m$. This allows us to obtain a list of permutations in $IS_{n+4}$ from those in $IS_n$ and $IS_4$.
	% Next, we investigate the use of 4-dimensional permutations in extending \( n \)-dimensional permutations to \( (n+4) \)-dimensional permutations compatible with $I_{n+4}$.
	% Concretely, when the quadratic bent function \( f(x) = Q_{\pi}(x) + Q_{I_n}(x) \) meets either the Walsh-Hadamard condition (WHC) or the Negated Walsh-Hadamard condition (NWHC), as defined in Equations (\ref{whc}),
	%and (\ref{nwhc}), 
	% the permutation \( \pi \) in $IS_n$ can be extended to a permutation  in $IS_{n+4}$.
	% Chunlei: I removed this sentence as it appears a bit redundant and detailed in the context. In addition, it is not precisely consistent with Th. 3.5.
	Furthermore, we show that all compatible sets in dimension $4$ can be recursively extended to a \cli compatible \cl set of the same size in dimension \( 4m \) for any \( m \geq 2 \). 
	Consequently, this yields NOMA codebooks of $6\cdot 2^{4m}$ GDJ sequences with 
	%(Pante):I almost thought this is wrong, 
	%until I remembered Yu's result.
	optimally low coherence $2^{-2m}$, which complements the result in~\cite{vik2024}.
	\cn
	% A computer search indicates that the largest compatible set of permutations in dimension \(4\) contains \(6\) elements. Following our proposed construction, this set can be extended recursively to obtain a set of \(6\) compatible permutations in any dimension \( 4m \).
	
	The remainder of this paper is organized as follows: Section~\ref{Sec:prelim} recalls
	basic and auxiliary results, and proposes the main research problems. Section~\ref{Sec:Construction} first gives comprehensive results for compatible permutations in dimension~$4$, then presents a recursive approach to constructing NOMA codebook for dimensions~$4m$. The work is concluded in Section~\ref{Sec:Con}.

	\section{Preliminaries}
	\label{Sec:prelim}
	Below we recall some basics of Boolean functions~\cite{CarletBook21}, spreading matrices in the context of NOMA schemes~\cite{yulnomacodebook}, and \cli present \cn the main research problem in the design of NOMA codebooks.
	\subsection{Boolean functions}
	\label{subsec:BF}
	Let $n$ be a positive integer. Denote by $\F_2^n$ the $n$-dimensional vector space over the finite field $\F_2$.
	There is a natural one-to-one correspondence $\varphi$ between the set $\{0, 1, \dots, 2^n-1\}$ and $\F_2^n$, namely, for an integer $j$ with $0\leq j<2^n$, one has $\varphi(j) = (j_1, \dots, j_n) \in \F_2^n$ with $j=j_1 + j_2\, 2 + \dots + j_n\, 2^{n-1}$. We shall identify an integer $j\in \{0, 1,\dots, 2^n-1\}$ as its corresponding vector $\varphi(j)$ and use them interchangeably when the context is clear. 
    \cli For $k=1,2,\dots, n$, we will use $e_k$ to denote the $k$-th row of the $n\times n$ identity matrix over $\F_2$. \cn
	
	An $n$-variable Boolean function is a function from $\F_2^n$ to $\F_2$. It can be represented by the \textit{truth-table} as $(f(0),f(1),\dots, f(2^n-1))$ or by the unique \textit{algebraic normal form} as 
	\[
	f(x) = \sum_{I \subseteq\{1, \ldots, n\}} a_I\left(\prod_{i \in I} x_i\right),
	\] 
	where the sum is taken modulo $2$ and $x=(x_1,x_2,\dots, x_n)\in \F_2^n$.
	The \textit{algebraic degree} of $f(x)$ is defined by $\deg(f) = \max \left\{ \lvert I\rvert\,:\, a_I \neq 0\right\}$, where $\vert I \vert$ denotes the size of $I$ (with the convention that the zero function has algebraic degree~0).
	An $n$-variable Boolean function is said to be \textit{linear} when it is of the form $L_c(x)=c_1x_1 + c_2x_2+\dots + c_nx_n$ for a vector $c=(c_1,c_2,\dots,c_n)\in \F_2^n$, and is said to be \textit{quadratic} when its algebraic degree is two.
	
	\begin{defn}
		\label{walsh_bent}
		The Walsh-Hadamard transform of an $n$-variable Boolean function $f(x)$ at a point $c \in \F_2^n$ is given by \ccb
		\[
		W_f(c) = \sum_{x\in \F_2^n}(-1)^{f(x) +  L_c(x) }.
		\]
		The set of Walsh-Hadamard coefficients $W_f(c)$ for all $c\in \F_2^n$ is called the \textit{Walsh-Hadamard spectrum} of $f(x)$. 
		The function $f(x)$ is said to be \textit{bent} if its Walsh-Hadamard spectrum is $\{\pm 2^{\frac{n}{2}}\}$ for even $n$, and said to be \textit{near-bent}  if its Walsh-Hadamard spectrum is $\{0, \pm 2^{\frac{n+1}{2}}\}$ for odd $n$. \ccn
	\end{defn} 
	   For a quadratic $n$-variable Boolean function $Q(x)$, its \textit{bilinear mapping} is given by $B(x,y)= Q(x+y)+Q(x)+Q(y)$. \cli 
    The kernel of the bilinear mapping of $Q(x)$ is defined by $V_Q = \{y\in \F_2^n\,:\, B(x,y) = 0 \text{ for any } x\in \F_2^n\}$, and the rank of $Q(x)$  is given by $r = n-\dim_{\F_2}(V_Q)$ ($\dim_{\F_2}(V_Q)$  denotes the dimension of the vector space $V_Q$ over $\F_2$). \cn  In addition, the bilinear mapping $B(x,y)$ of a quadratic function $Q(x)$ can be characterized by its corresponding symplectic matrix $\bB$, which is defined as an $n\times n$ binary matrix such that for $1\leq i, j\leq n$, the entry $\bB(i,j) = 1$ if and only if  
	$B(x,y)$ contains the term  $x_iy_j$, or equivalently, $x_ix_j$ occurs in the quadratic function $Q(x)$.
	Consequently, the rank of a quadratic function $Q(x)$ is identical to the rank of the corresponding symplectic matrix $\bB$~\cite{MacWilliams1983}. 
	It is a well-known fact~(see~\cite[p. 441]{MacWilliams1983} or \cite[Ch. 16]{SM16}) that the Walsh-Hadamard spectrum of $Q(x)$ depends upon  its rank only, precisely,  
	\begin{equation}\label{Eq_quad}
		W_Q(c) = \sum\limits_{x\in \F_2^n}(-1)^{Q(x) + L_c(x) } = 
		\begin{cases}
			\pm 2^{n-\frac{r}{2}}, &\text{ if $Q(x)=0\, \text{ for all }\, x\in V_Q$,} 
			\\
			0, & \text{otherwise}.
		\end{cases}
	\end{equation}
	
\cl 	For a quadratic bent function, we introduce a condition that will be used in subsequent discussions.
	
	\begin{defn}\label{Def_WHC}
		Let $n\geq 4$ be an even integer, and let $i,\,j$ be integers with $1\leq i<j \leq n$. A quadratic bent function $Q(x)$  on $\F_2^n$ 
		is said \cli to satisfy the \textit{Walsh-Hadamard condition (WHC)} \cl on $(i,\,j)$ if 
		\[
		\prod_{\alpha, \beta \in \F_2}W_{Q}(c+\alpha e_i+\beta e_j)  = 2^{2n}, \quad \cli \forall \,\, c\in \F_2^n,  \cl 
		\] \cli 
		{\color{cyan} where $e_i$ and $e_j$ denote the $i$-th row and $j$-th row, respectively, of the $n\times n$ identity matrix over $\F_2$.}

        \cn
	\end{defn}

	% Two Boolean functions $f$ and $g$ in $n$ variables are said to be EA-equivalent, 
	% if there  exists $A$, an affine automorphism of $\mathbb{F}_2^n$, $\ell$ an affine Boolean function in $n$ variables,
	% such that $f=g\circ A+\ell$.
	% A quadratic function $Q$ with rank~$r$ is EA-equivalent to a quadratic function of the form~$T(y)=\sum_{i=1}^r y_{2i-1}y_{2i}$, (this is often called Dickson's Theorem, see~\cite[Chapter 15]{MacWilliams1983}). 
	% \ccb where $V_Q$ is the kernel of the bilinear linear mapping $B(x,y)$ of $Q$ defined by $B(x,y)= Q(x+y)+Q(x)+Q(y)$, namely,
	% $V_Q= \{y \in \F_2^n \mid B(x,y)=0 \text{ for all }\,x\in \F_2^n\}$, and the rank of $Q$ is given by $r = n-\dim_{\F_2}(V_Q)$. \ccn

	% For basics on Boolean functions, namely functions from $\F_2^n$ to $\F_2$,  the reader may refer to the monographs~\cite{CarletBook21,CS17}. Below we only recall some notions that will be used often.

	\subsection{NOMA codebooks from complementary sequences}
	\label{subsec:CS}
	
	\ccn
	% \Comment{\textcolor{red}{I think we should remove boldfaced vectors, to be consistent}
	% We will use the inner product in the definition of coherence, which are concerned with sequences instead of vectors.
	% }
	
	    For a complex-valued sequence $\mathbf{a}=(a_0,\ldots,a_{N-1})$, its {\em aperiodic autocorrelation} $C_{\mathbf{a}}(\tau)$ at a shift $\tau$, where $\tau$ is an integer with $\vert\tau\vert<N$, is defined as follows: 
	\[C_{\mathbf{a}}(\tau)=
	\begin{cases}
	\sum_{i=0}^{N-\tau-1}a_i a^*_{i+\tau}, & 0\leq \tau<N,\\
	\sum_{i=0}^{N+\tau-1}a_{i-\tau} a^*_i,&-N<\tau<0,    
	\end{cases}
	\]where $a^*_i$ denotes the complex conjugate of $a_i.$
	% It is readily seen that $C_{\mathbf{a}}(0)$ equals the 
	% inner product of $\ba$ given by $\langle \ba, \ba \rangle = \sum_{i=0}^{N-1}a_ia^*_i$.
	A pair of sequences $\mathbf{a},\mathbf{b}$ of length $N$ is called a {\em Golay complementary pair}~\cite{golay1949} if they satisfy
	\begin{equation}\label{Eq_ComplPair}
		C_{\mathbf{a}}(\tau)+C_{\mathbf{b}}(\tau)=0
	\end{equation} \cli for any integer $\tau$ with $0<\vert \tau \vert <N$. \cn 
    \cn 
	Each sequence of such a complementary pair is called a
	{\em Golay complementary sequence}.
	\begin{defn}
		\label{def:papr}
		For a unimodular sequence $\mathbf{a}=(a_0,a_1,\dots, a_{N-1})$, 
        \cli where the modulus $\vert a_i \vert =1$, for $0\leq i<N$, \cn
		when it is transmitted through $N$ subcarriers, 
		the \textit{peak-to-average power ratio} (PAPR) of its OFDM signal
		is defined by 
		$$
		{\rm PAPR}(\mathbf{a}) = 
		\frac{\max_{t\in [0,1)} {\left\vert\sum_{i=0}^{N-1}a_ie^{2\pi\sqrt{-1}\,it}\right\vert}^2}{N}.
		$$  
	\end{defn}
	Note that $$\left\vert\sum_{i=0}^{N-1}a_ie^{2\pi\sqrt{-1}\,it}\right\vert^2 = 
	N + \sum_{0<\vert \tau \vert < N}C_{\mathbf{a}}(\tau) e^{2\pi \sqrt{-1} t\tau }.
	$$ 
	For a pair of Golay complementary pair $(\ba,\,\bb)$, due to their complementary property as in \eqref{Eq_ComplPair}, one obtains ${\rm PAPR}(\mathbf{a}) + {\rm PAPR}(\mathbf{b}) =  2$. 
    This indicates that each Golay complementary sequence has PAPR upper bounded by 2 since the value of PAPR is always non-negative. Owing to this attractive property, Golay complementary sequences have been widely used for PAPR control in multi-carrier OFDM schemes. In 1999 Davis and Jedwab \cite{Davis1999} established a fundamental relation between binary Golay complementary pairs of length $2^n$ and quadratic Boolean functions of particular forms as below. 
    % \cli Binary Golay complementary sequences of length $2^n$ are later referred to as the Golay-Davis-Jedwab (GDJ) sequences in the literature. \cn
	
	\begin{lemma}[\textup{\cite[Th.\,3]{Davis1999}}] 
		\label{lem_GDJ}
		Let $\pi$  be  a permutation of $\{1,2, \dots, n\}$, $c=(c_1,c_2, \dots, c_n) \in \F_{2}^n$ and 
		$$f_{\pi}^c\left(x\right)=Q_{\pi}(x) + L_c(x) =\sum_{i=1}^{n-1} x_{\pi(i)} x_{\pi(i+1)}+ \sum_{i=1}^n c_i x_i.$$
		For $\epsilon, \epsilon^{\prime} \in \F_{2}$, define the functions
		\[
		\begin{split}
		& a\left(x\right)=f_{\pi}^c\left(x\right)+\epsilon, \\
		& b\left(x\right)=f_{\pi}^c\left(x\right)+x_{\pi(1)}+\epsilon^{\prime}.
		\end{split}
		\] \cl Let $\ba=(a_0,a_1,\dots, a_{2^n-1}), \bb=(b_0,b_1,\dots, b_{2^n-1})$ be two sequences associated with the functions $a(x)$ and $b(x)$ given by $a_j = (-1)^{a(j)}$ and $b_j=(-1)^{b(j)}$, respectively, for $0\leq j<2^n$. Then the sequences $\ba$ and $\bb$ form a binary Golay complementary pair of length $2^n$.  \cn
	\end{lemma} 
	% Each complementary sequence in Lemma \ref{lem_GDJ} is termed a GDJ sequence in the literature. In the context of Reed-Muller codes, given a permutation $\pi$, the set 
	% $\left\{Q_{\pi}(x) + L_c(x) + \epsilon \mid c \in \F_2^n, \epsilon \in \F_2\right\}$
	% is a coset of the first-order RM code $\mathcal{RM}(1,n)$, with representative $Q_{\pi}(x)$, in the second-order RM code $\mathcal{RM}(2,n)$. 

	% Given a Boolean function $f$ in $n$ variables,  
	% one can define the (sign) sequence of length $2^n$ over $\{\pm 1\}$ as $\mathbf{f} = ((-1)^{f(0)}, (-1)^{f(1)}, \dots, (-1)^{f(2^n-1)})^\intercal$,
	% which is said to be associated with the function~$f$. 

	\ccb 
	% For a permutation $\pi$ of $\{1,2,\dots, n\}$, and $x=(x_1,\ldots,x_n)\in \F_2^n$, let $Q_{\pi}(x)$ be 
	% a quadratic Boolean function given by
	% $Q_{\pi}(x) =\sum_{i=1}^{n-1} x_{\pi(i)} x_{\pi(i+1)}$,
	% which forms \ccb a Hamiltonian path on the complete graph $K_n$ by drawing edges between vertices \ccn $x_{\pi(i)}$ and  $x_{\pi(i+1)}$ for $i=1, 2,\dots, n-1$. 
	% %\Comment{\textcolor{red}{I suggest removing this %sentence about the Hamiltonian path, as it requires %more explanation but adds little value in the context.}}
	% The linear Boolean function with coeffients in $c=(c_1,c_2, \dots, c_n)\in \F_2^n$ will be denoted by $L_c(x) = \sum_{i=1}^n c_i x_i$. 

	% The above relation has stimulated quite a lot of research on complementary sequences sets with low PAPR. \ccb 
	% The complementary pair $(\ba, \bb)$ in ve lemma is also termed a Golay-Davis-Jedwab (GDJ) pair, and 
	For compressed sensing-based joint channel estimation and multiuser detection in uplink grant-free NOMA, Yu \cite{yulnomacodebook} proposed a codebook in the form of a spreading matrix, \cli which should have low coherence to minimize interference and
    a large number of columns to accommodate sufficiently many user devices. \cn 
	We first recall the \textit{coherence} of a spreading matrix.
	\begin{defn}
		\label{def_coherence}
		Given an $N\times K$ matrix $\bPhi$ over the complex field $\mathbb{C}$, the \textit{coherence} of $\bPhi$ is given by 
		\[
		\mu(\bPhi) = \max_{1\leq k_1\neq k_2\leq K}\frac{\left\vert\langle \ba_{k_1}, \ba_{k_2}\rangle\right\vert}{\|\mathbf{a}_{k_1}\|_2\, \|\mathbf{a}_{k_2}\|_2},
		\] 
		where \cli $\ba _{k_1}, \ba_{k_2}$ are the $k_1$-th, $k_2$-th columns of $\bPhi$, respectively, and the notation \cn
		$\langle \ba, \bb\rangle =  \sum_{i=0}^{N-1}a_ib^*_i$ denotes the inner product between two sequences $\ba$ and $\bb$, \cli 
		and $\| \ba \|_2 = \sqrt{\langle \ba, \ba \rangle}$ is the $L^2$-norm of a sequence $\ba$. The ratio $K/N$ is called the overloading factor of the matrix $\bPhi$.
        \cn
	\end{defn}

	In \cite{yulnomacodebook}, Yu proposed a framework for designing the uplink grant-free NOMA scheme based on the complementary sequences in Lemma \ref{lem_GDJ}, referred to as the Golay-Davis-Jedwab (GDJ) sequences in the literature. 
	Here we recall the basics of the framework in \cite{yulnomacodebook}. 
	Arrange all the GDJ sequences associated with the quadratic functions $f_{\pi}^c(x)=Q_{\pi}(x)+L_c(x)$, where $c\in \F_2^n$, in a spreading matrix column by column, where the columns are indexed by $c$. Then we obtain a $2^n\times 2^n$ spreading matrix, \cli in which any two columns indexed by $c_1,c_2$ are orthogonal (since their inner product is equal to $\sum_{x\in \F_2^n}(-1)^{L_{c_1+c_2}(x)} $, which vanishes at any different indices $c_1, c_2\in \F_2^n$). \cn  Observe that for any sequence $\textbf{s}$ associated with the function
	$a(x)=f_{\pi}^c(x)+1$ or $b(x)=f_{\pi}^c(x)  + x_{\pi(1)} +  \epsilon'$ with $\epsilon' \in \F_2$ as in Lemma~\ref{lem_GDJ}, \cli the column $\textbf{s}'$ indexed by the vector $c$ or $c+e_{\pi(1)}$ (which differs from $c$ at the $\pi(1)$-th position) in the existing spreading matrix satisfies  $\vert \langle \mathbf{s}', \textbf{s}\rangle \vert = 2^n$, which is the highest coherence. This indicates that one cannot add 
    more GDJ sequences derived from the same permutation $\pi$ into the existing spreading matrix. 
	To further widen \cn the spreading matrix while maintaining low coherence, Yu~\cite{yulnomacodebook} adopted more permutations of $\{1,2,\dots, n\}$ and proposed the following spreading matrix. \ccn
	\begin{defn} 
		\label{Def_Codebook_Design}
		Consider $L$ distinct permutations $\pi_1, \ldots, \pi_L$ of $\{1, \ldots, n\}$. 
        % For each~$\pi_{\ell}$ with \cli $1\leq \ell \leq L$, \cn let $\ba_{\pi_{\ell}}^{(c)}$ be the 
		% GDJ sequence of length $N=2^n$ associated with the function $f_{\pi_{\ell}}^c(x)$ as in Lemma~\textup{\ref{lem_GDJ}}.
		Let $N=2^n$ and define an $N \times L N$ non-orthogonal spreading matrix as follows:
		\begin{equation}\label{Eq_SpreadingMatrix}
			\bPhi=\frac{1}{\sqrt{N}}\left[\bPhi_1, \ldots, \bPhi_L\right],
		\end{equation}
		where for each $1\leq \ell \leq L$,
		$$
		\bPhi_\ell=\left[\ba_{\pi_\ell}^{(0)}, \ba_{\pi_\ell}^{(1)}, \ldots, \ba_{\pi_\ell}^{(N-1)}\right]_{N\times N}
		$$ is a column-wise orthogonal matrix consisting of GDJ sequences  \cli $\ba_{\pi_\ell}^{(c)}$ associated with the $n$-variable quadratic functions $f_{\pi_{\ell}}^c(x)=Q_{\pi_\ell}(x) + L_c(x)$ as in Lemma~\textup{\ref{lem_GDJ}}, where $c$ runs through all vectors in $\F_2^n$. \cn
	\end{defn}
	
	The coherence of the above spreading matrix can be characterized in the following way~\cite[Th.\,1]{yulnomacodebook}.
	\begin{lemma}\label{lem_coherence}
		Let  $\bPhi$ be an $N\times LN$ spreading matrix defined as in \eqref{Eq_SpreadingMatrix}. \cli For %$1\leq \ell_1,\ell_2\leq L$,
        $\pi_{\ell_1},\pi_{\ell_1}\in\{\pi_1,\ldots,\pi_L\}$,
         let $\bB_{\ell_1,\ell_2}$ be an $n\times n$ binary symplectic matrix
         defined as follows:
		  for $1\leq i, j\leq n$, $\bB_{\ell_1,\ell_2}(i,j)=1$ if and only if  the 
		quadratic form $%Q_{\pi_{\ell_1}, \pi_{\ell_2}}(x)=
        Q_{\pi_{\ell_1}}(x) + Q_{\pi_{\ell_2}}(x)$ \cl contains the term $x_ix_j$. Then the coherence of the spreading matrix $\bPhi$
		is given by
		\begin{equation}\label{Eq_coherence}
			\mu(\bPhi ) = \frac{1}{\sqrt{2^{r_{\min}}}}, \text{ where } r_{\min }=\min _{1 \leq \ell_1 \neq \ell_2 \leq L} \operatorname{rank}\left(\mathbf{B}_{\ell_1, \ell_2}\right).
		\end{equation}
	\end{lemma}
	To have a low coherence of the above spreading matrix, it is desirable to keep $r_{\min}$ as large as possible. Since each matrix $\bB_{\ell_1,\ell_2}$ is a skew-symmetric matrix over $\F_2$, \cli which has the maximum rank $n$ for even $n$ (resp., $n-1$ for odd $n$), \cn 
	the coherence $\mu(\bPhi )$ has the following lower bounds: 
	\begin{equation}
		\label{Eq_coh_lowerbounds}
		\mu(\bPhi )  \geq \begin{cases}
			\frac{1}{\sqrt{2^{n}}} \text{ for even } n, \\
			\frac{1}{\sqrt{2^{n-1}}} \text{ for odd } n.
		\end{cases}
	\end{equation}
	From Definition \ref{Def_Codebook_Design}, 
 %    it can be verified that 
	% \ccb $
	% \langle \ba_{\pi_{\ell_1}}^{(c_1)},  \ba_{\pi_{\ell_2}}^{(c_2)} \rangle 
	% = W_{f}(c_1+c_2)
	% $ \ccn with the function $f(x)=Q_{\pi_{\ell_1}}(x) + Q_{\pi_{\ell_2}}(x)$.
	% Therefore, 
    the coherence of the spreading matrix $\bPhi$ in \eqref{Eq_SpreadingMatrix} can be characterized by the Walsh-Hadamard transforms of the quadratic Boolean functions $%Q_{\pi_{\ell_1},\pi_{\ell_2}}(x) = 
    Q_{\pi_{\ell_1}}(x) + Q_{\pi_{\ell_2}}(x)$.  
	In the language of the Walsh-Hadamard transform, for $1\leq \ell_1 \neq \ell_2 \leq L$, the quadratic function $%Q_{\pi_{\ell_1},\pi_{\ell_2}}(x)
    Q_{\pi_{\ell_1}}(x) + Q_{\pi_{\ell_2}}(x)$ in Lemma \ref{lem_coherence} with rank $r$ has Walsh-Hadamard spectrum $\{0, \pm 2^{n-\frac{r}{2}}\}$. 
	Define \begin{equation}
		\label{Eq_Phi_Walsh}
		W(\bPhi) = \max_{1\leq \ell_1\neq \ell_2 \leq L} \max_{c\in \F_2^n} \mid W_{Q_{\pi_{\ell_1}}+Q_{\pi_{\ell_2}}}(c)\mid.
	\end{equation} 
	Then we have 
	$$
	W(\bPhi) = \frac{2^n}{\sqrt{2^{r_{\min}}}},
	{\text{ where }}r_{\min }=\min _{1 \leq \ell_1 \neq \ell_2 \leq L} \operatorname{rank}\left(Q_{\pi_{\ell_1}}(x) + Q_{\pi_{\ell_2}}(x)\right),
	$$ which yields the relation
	\(
	\ccb W(\bPhi) = 2^n \mu(\bPhi). \ccn
	\)
	\ccb Equivalently, the equality in \eqref{Eq_coh_lowerbounds} is achieved when the quadratic function $Q_{\pi_{\ell_1}}+Q_{\pi_{\ell_2}}(x)$, for any $1\leq \ell_1\neq \ell_2 \leq L$, is a bent function for even $n$, and a near-bent function for odd $n$.

	\subsection{The main research problem}

    \cli

    Let $\mathbf{\Psi}$ be an $N\times K$ complex-valued matrix where each column has $L^2$ norm $\sqrt{N}$. According to the well-known Welch bound \cite{Welch1974}, the coherence of $\mathbf{\Psi}$ satisfies
        \[
         \mu(\mathbf{\Psi}) \geq \sqrt{\frac{K-N}{(K-1)N}}.
        \]
        Matrices that achieve the equality in the Welch bound  have various applications in communications. However, it is rather challenging to construct them for $K\geq N+2$ (see~\cite{cding2007,Sustik2007} and references therein).
        % It is known that the above lower bound is achieved if and only if the 
        % matrix $\mathbf{\Psi}$, after each column being normalized to have norm one,
        % is an equiangular tight frame \cite[Th.~C]{Sustik2007}.  Equiangular tight frames       
        % Sustik et al. in \cite[Th.~C]{Sustik2007} showed that when the entries of $\mathbf{\Psi}$ are real, equiangular tight frames exist only if 
        % $K\leq \min\{N(N+1)/2, (K-N)(K-N+1)/2\}$. 
        In particular, when $N=2^n$ and the entries of $\mathbf{\Psi}$ take values $\pm 1$, each column of $\mathbf{\Psi}$ can be associated with a certain $n$-variable Boolean function. In this case, the coherence of $\mathbf{\Psi}$ can be characterized by the Walsh-Hadamard transform at the zero point for $n$-variable Boolean functions. As a result, when $K>N$,  the coherence of $\mathbf{\Psi}$ satisfies $\mu(\mathbf{\Psi})\geq \frac{1}{\sqrt{N}}$, 
        since for any nonlinear $ n$-variable Boolean function, 
        the maximum magnitude of its Walsh-Hadamard coefficients is at least $2^{\frac{n}{2}}$~\cite{CarletBook21}. 
        Furthermore, when columns of $\mathbf{\Psi}$ are associated with quadratic Boolean functions, the coherence of $\mathbf{\Psi}$ for odd $n$ satisfies $\mu(\mathbf{\Psi}) \geq \sqrt{\frac{2}{{N}}}$ as in \eqref{Eq_coh_lowerbounds}. 
        We see that the two lower bounds on the coherence of matrices $\mathbf{\Psi}$ that are generated from $n$-variable Boolean functions are close to the Welch bound $\sqrt{\frac{K-N}{(K-1)N}}$ for generic complex-valued matrices in the case of $K=LN$.

        In the application for uplink grant-free NOMA,  the design of the spreading matrix $\bPhi$ in Definition~\ref{Def_Codebook_Design} has several advantages: each column of $\bPhi$ as a spreading sequence has low PAPR upper bounded by $2$, and the coherence of $\bPhi$ can take a low value (close to the Welch bound) when the symplectic matrix $\bB_{\ell_1,\ell_2}$ for any $1\leq \ell_1 \neq \ell_2 \leq L$ has a
	largest possible rank, namely, $n$ for even $n$ or $n-1$ for odd $n$.
       In addition, the spreading matrix $\bPhi$ in \eqref{Eq_SpreadingMatrix} is desired to have an overloading factor $L=K/N$ as large as possible to accommodate sufficiently many devices. \cn 
        Interestingly, this problem was also raised earlier by Paterson \cite{Paterson2002} in the context of algebraic coding theory, where he considered GDJ sequences for OFDM codes and gave a relatively trivial upper bound on $L$ for even $n$.
	More specifically, for any quadratic function $Q_{\pi}(x)=\sum_{i=1}^{n-1} x_{\pi(i)} x_{\pi(i+1)}$, its corresponding symplectic matrix $\mathbf{B}$ with $\mathbf{B}(i,j)=1$ if and only if $x_ix_j$ occurs in $Q_{\pi}(x)$ has the following property: its first row has a zero as its first entry and Hamming weight either 1 or 2. This implies that there are at most $(n-1) + {n-1\choose 2} = {n\choose 2}$ different choices of the first row of $\mathbf{B}$. Since the sum of two symplectic matrices is required to have rank $n$ for even $n$, one can have at most ${n \choose 2}$ symplectic matrices with different first rows. Consequently, there are at most ${n \choose 2}$ permutations on $\{1,2,\dots, n\}$ that are mutually compatible. \cli This gives the following upper bound on the overloading factor of the spreading matrix $\bPhi$ in Definition \ref{Def_Codebook_Design} that has the lowest coherence $2^{-\frac{n}{2}}$:
    \begin{equation}\label{Eq_PatersonBound}
     L \leq n(n-1)/2.   
    \end{equation} \cn 
     This upper bound for $n=4$ can be easily confirmed to be tight. However, it appears far from being tight for large values $n\geq 6$ according to the experimental results as in~\cite[Table III]{vik2024}. 
	Here, natural, albeit challenging, problems arise.
	\begin{prob}
		\label{prob1}
		\cli	Let $N=2^n$ with $n\geq 6$ and $\bPhi$ be the $N\times LN$ spreading matrix given in Definition~\textup{\ref{Def_Codebook_Design}}.  
        Give a tighter upper bound on the overloading factor $L$ such that 
         the spreading matrix $\bPhi$ can maintain the lowest coherence 
		\[\mu(\bPhi ) = \begin{cases}
		\frac{1}{\sqrt{2^{n}}} \text{ for even } n, \\
		\frac{1}{\sqrt{2^{n-1}}} \text{ for odd } n.
		\end{cases}
		\]
        Furthermore, construct a spreading matrix $\bPhi$ with its overloading factor $L$ approaching the upper bound for infinitely many $n$. 
		% \end{enumerate} 
	\end{prob}

	%In the sequel, we shall provide a theoretical construction of such a spreading matrix with an overloading factor $L=6$ for infinitely many $n$.
	\ccn
	
	\section{A Construction by Permutation Extension}
	\label{Sec:Construction}
	
	In this section, we will be concerned with theoretical constructions of a set of permutations of $\{1,2,\dots, n\}$ that
	satisfy the condition in our Main Problem~\ref{prob1}. 
	
	Denote by $I_n$ the identity permutation and by 
	$S_n$ the set of all permutations of $\{1, 2, \ldots, n\}$. 
	As is customary, we write a permutation as $\pi=[i_1,i_2,\ldots,i_n]$ 
	% or   $\pi=i_1i_2\cdots i_n$ (when there is no ambiguity) 
	to mean that $\pi$ is defined by $\pi(1)=i_1,\pi(2)=i_2,\hdots, \pi(n)=i_n$. Given a permutation $\pi=[i_1,i_2,\ldots,i_n]$, we refer to the permutation $\widetilde{\pi}=[i_n,i_{n-1},\ldots,i_1]$ as its {\em reverse}.
	
	\smallskip
	
	We introduce the notion of \textit{compatible permutations} below.
	\begin{defn} 
		Two permutations $\pi_1, \,\pi_2$ in $S_n$ are said to be \textit{compatible}, denoted by $\pi_1 \sim \pi_2$,
		if the corresponding quadratic function $\cli Q_{\pi_1}(x)+Q_{\pi_2}(x) = \cn 
        \sum_{i=1}^{n-1}x_{\pi_1(i)}x_{\pi_1(i+1)}+\sum_{i=1}^{n-1}x_{\pi_2(i)}x_{\pi_2(i+1)}$ is bent for even $n$ and near-bent for odd $n$.  
		In addition, a subset $S\subset S_n$ is called a \textit{compatible set} if any two permutations in $S$ are compatible. 
	\end{defn}   
	It is desirable to have a compatible set with set size as large as possible. We first give some basic properties for compatible permutations.
	
	\begin{lemma}
		\label{lem_properties} 
		Let $\widetilde{\pi}$, $\pi^{-1}$ denote, respectively, the reverse and compositional inverse of a permutation $\pi$
		in $S_n$. Then for any two permutations $\pi, \, \sigma \in S_n$, we have:
		\begin{enumerate}[label=(\roman*)] 
			\ccb  \item[$(i)$] $\sigma\sim \pi$ if and only if \cl $\sigma\sim \widetilde{\pi};$\cn
			\item[$(ii)$]  $I_n \sim \pi$ if and only if   $I_n \sim \pi^{-1};$ 
			\cl		\item[$(iii)$]  $\pi \sim \sigma$  if and only if  $I_n \sim \sigma^{-1}\circ \pi$, where
			$\circ$ is the composition operator. \cn
		\end{enumerate}
	\end{lemma}
	\begin{proof}
		It is easy to see that $Q_\pi(x)=Q_{\widetilde{\pi}}(x)$, so \ccb $Q_{\sigma}(x)+Q_{\pi}(x)$  and $Q_{\sigma}(x)+Q_{\widetilde{\pi}}(x)$
		represent the same function, which implies claim~$(i)$\ccn. It can be observed that
		$Q_{I_n}(x_1,\ldots,x_n)+Q_{\pi}(x_1,\ldots,x_n)=Q_{\pi^{-1}}(x_{\pi(1)},\ldots,x_{\pi(n)})+Q_{I_n}(x_{\pi(1)},\ldots,x_{\pi(n)})$, which means that 
		$Q_{I_n}+Q_{\pi}$ and $Q_{I_n}+Q_{\pi^{-1}}$ have the same Walsh-Hadamard \ccb spectrum. \ccn 
		This implies claim~$(ii)$. Finally,
		with $(y_1,\dots, y_n) = \left(x_{\sigma^{-1}(1)}, \dots, x_{\sigma^{-1}(n)}\right)$,
		we have 
		\[
		\begin{split}
		Q_{I_n}(x_1,\ldots,x_n)+Q_{\sigma^{-1}\circ \pi}(x_1,\ldots,x_n)
		&= \sum_{i=1}^{n-1}(x_ix_{i+1} + x_{\sigma^{-1}\circ \pi(i)}x_{\sigma^{-1}\circ \pi(i+1)})
		\\&=\sum_{i=1}^{n-1}(y_{\sigma (i)}y_{\sigma(i+1)} + y_{\pi(i)}y_{\pi(i+1)})
		\\&=
		Q_{\sigma}(y_1,\dots, y_n) + Q_{\pi}(y_1,\dots, y_n),
		\end{split}
		\]
		which implies  claim~$(iii)$ when the Walsh-Hadamard transforms of the above quadratic functions are considered.
	\end{proof}
	
	\cl By Lemma~\ref{lem_properties} $(i)$, it is readily seen that the relations
	$\pi \sim \rho$, $\pi \sim \widetilde{\rho}$, $\widetilde{\pi}\sim \rho$
	and $\widetilde{\pi} \sim \widetilde{\rho}$ are equivalent.
	In addition, it follows from Lemma~\ref{lem_properties} $(ii)$ and $(iii)$ that $\pi \sim \sigma$  if and only if  $I_n \sim  \pi^{-1}\circ \sigma$ since  $\pi^{-1}\circ \sigma = (\sigma^{-1}\circ \pi)^{-1} $. \cn
	Without loss of generality, we can always assume that a compatible set $S \subset S_n$ contains the identity permutation $I_n$ (since one can obtain $I_n$ by applying the inverse of a permutation to all the others in a compatible set). From now on, we will focus only on compatible sets containing $I_n$. For ease of presentation, 
    \cli we denote by $IS_n$ the set of all permutations in $S_n$ that are compatible with $I_n$, i.e.,
	\[
	IS_n = \left\{\pi \in S_n \, \vert\, \pi \sim I_n\right\}.
	\]
    \cn 
	To gain some insight into compatible sets, we first take a look at 
	the case of $n=4$.
	
	\subsection{Compatible sets for $n=4$}
	\label{subsect_n4}
	
	When $n=4$, by an exhaustive search we obtain all the permutations that are compatible with the identity permutation $I_4$ as below:
	\[
	\begin{array}{rrrr}
	\rho_1=[3,2,4,1], &
	\rho_2=[2,4,1,3], &
	\rho_3= [ 3, 4, 1, 2 ], &
	\rho_4= [ 2, 4, 3, 1 ], \\
	\rho_5= [ 3, 1, 4, 2 ], &
	\rho_6=[ 1, 3, 4, 2 ],  &
	\rho_7=[ 4, 2, 1, 3 ], &
	\rho_8= [ 2, 1, 4, 3 ],\\
	\rho_9=[ 4, 1, 3, 2 ], &
	\rho_{10}=[ 2, 3, 1, 4 ], &
	\rho_{11}=[ 1, 4, 2, 3 ],&
	\rho_{12}=[ 3, 1, 2, 4 ].
	\end{array}
	\]
	Among these permutations, it is easily seen that 
	\begin{equation}
		\label{reverses}
		\rho_{11}=\widetilde{\rho_1}, \,\rho_5=\widetilde{\rho_2}, \,\rho_8=\widetilde{\rho_3},\, \rho_6=\widetilde{\rho_4}, \,\rho_{12}=\widetilde{\rho_7}, \rho_{10}=\widetilde{\rho_9}
	\end{equation} 
	and 
\begin{equation}
\label{inverses}
\begin{split}
\rho_5&=\rho_2^{-1},\,\rho_3^{-1}=\rho_3, \,\rho_8=\rho_8^{-1},\\
\rho_7&=\rho_1^{-1}=\rho_1^2,\, \rho_9=\rho_4^{-1}=\rho_4^2, 
	\rho_{11}=\rho_6^{-1}=\rho_6^{2},  \,\rho_{12}=\rho_{10}^{-1}=\rho_{10}^2.
		\end{split}
	\end{equation}
	By Lemma \ref{lem_properties} $(i)$-$(ii)$ and Equations \eqref{reverses}--\eqref{inverses},  the following \cli three chains \cn of permutations are compatible with $I_4$:
	\[
	\begin{array}{l}
	\rho_1 \xleftrightarrow[]{\text{rev}} \rho_{11} \xleftrightarrow[]{\text{inv}} \rho_{6}
	\xleftrightarrow[]{\text{rev}} \rho_4 \xleftrightarrow[]{\text{inv}} \rho_9 
	\xleftrightarrow[]{\text{rev}} \rho_{10} \xleftrightarrow[]{\text{inv}} \rho_{12} 
	\xleftrightarrow[]{\text{rev}} \rho_{7} \xleftrightarrow[]{\text{inv}} \rho_{1},
	\\
	\rho_2
	\xleftrightarrow[]{\text{rev}} \rho_{5} \xleftrightarrow[]{\text{inv}} \rho_{2}, \\
	\rho_3
	\xleftrightarrow[]{\text{rev}} \rho_{8} \xleftrightarrow[]{\text{inv}} \rho_{8} \xleftrightarrow[]{\text{rev}} \rho_3. 
	\end{array}
	\]
	By Lemma \ref{lem_properties} $(iii)$, \ccb $\pi\sim\pi^{-1}$ if and only if $I_n\sim\pi^2$, where $\pi^2=\pi\circ\pi$. Then, by Equation~\eqref{inverses}, $\rho_1\sim\rho_7, \rho_4\sim \rho_9, \rho_6\sim\rho_{11}, \rho_{10}\sim\rho_{12}$, creating four examples of compatible sets of size 3. \ccn   
	More work is needed to get a compatible set of larger size in which any two permutations are compatible.
	Starting from each of these pairs, one needs to 
	find a new permutation compatible with the existing ones from the remaining permutations, and repeat this process to obtain maximal compatible sets for $n=4$. 
	
	\cl 
	For instance, starting from $\rho_1, \rho_7$, one can first add $\rho_3$ 
	since the permutations $\rho_1^{-1}\circ\rho_3 = \rho_7\circ\rho_3 = \rho_6$ and $\rho_{7}^{-1}\circ\rho_3=\rho_1\circ\rho_3 = \rho_9$ are compatible with $I_4$. 
	This gives us a compatible set $\{I_4,  \rho_1, \rho_3, \rho_7\}$. 
	To add a new permutation, say $\rho$, we need to choose $\rho\in IS_4$ and check whether it is compatible with the existing permutations $\rho_1, \rho_3, \rho_7$.
	
	To facilitate the calculation, we provide Table~\ref{tab_invcomp} to look up  $\rho_i^{-1}\circ\rho_j$ for $1\leq i,j\leq 12$, where ``$0$'' indicates $\rho_i^{-1}\circ \rho_j = I_4$, ``$-$'' indicates $\rho_i^{-1}\circ \rho_j \not\in \{\rho_1, \rho_2, \dots, \rho_{12}\}$  (i.e., $\rho_i^{-1}\circ \rho_j$ is not compatible with $I_4$),  and $k$ in the entry of the $i$-th row and $j$-th column in Table~\ref{tab_invcomp}  indicates $\rho_i^{-1}\circ \rho_j = \rho_k$. 
	Since 
	$
	(\rho_i^{-1}\circ \rho_j)^{-1} = \rho_j^{-1} \circ \rho_i
	$, Table \ref{tab_invcomp} gives all information about the compatibility between permutations in $IS_4$.
	From Table \ref{tab_invcomp}, for any $\rho_i, \rho_j$, one can easily verify
	whether $\rho_i^{-1}\circ \rho_j$ belongs to $IS_4$. It is readily seen that the two permutations $\rho_2, \rho_5$ are not compatible with any permutation in $IS_4$.
	Furthermore, since $Q_{\pi}(x)=Q_{\widetilde{\pi}}(x)$, no permutation $\pi$ can be compatible with its reverse $\widetilde{\pi}$.  
	
	Continuing with the compatible set $\{I_4,  \rho_1, \rho_3, \rho_7\}$, 
	we see that $\rho_6$ is compatible with the existing permutations in this set since $\rho_6^{-1}\circ \rho_1 = \rho_4$, $\rho_6^{-1}\circ \rho_3 = \rho_{10}$ and $\rho_6^{-1}\circ\rho_7=\rho_3$,
	yielding a compatible set $\{I_4,  \rho_1, \rho_3, \rho_7, \rho_6\}$. For the remaining permutations $\rho_2, \rho_4, \rho_5, \rho_8,\rho_9, \rho_{10}, \rho_{11}, \rho_{12}$, 
	from Table \ref{tab_invcomp} we see that $\rho_2, \,\rho_5$ should be excluded instantly, and that, due to 
	mutual incompatibility, $\rho_1$ excludes $\rho_{11}$, $\rho_3$ excludes $\rho_8$, $\rho_6$ excludes $\rho_4$, and $\rho_{7}$ excludes $\rho_{12}$.
	Hence, we can further add $\rho_9$ or $\rho_{10}$ to the compatible set, leading to 
	$\{I_4,  \rho_1, \rho_3, \rho_7, \rho_6, \rho_9\}$ or
	$\{I_4,  \rho_1, \rho_3, \rho_7, \rho_6, \rho_{10}\}$ \cli (as the first two in Table \ref{tab_CS_4}). \cl  Since $\rho_9$ is incompatible with $\rho_{10}$, no more permutations can be added to these two compatible sets. As a matter of fact, according to the upper bound \cli \eqref{Eq_PatersonBound} by Paterson~\cite{Paterson2002}, \cl they are maximal compatible sets for $n=4$.
	
	As shown in Table \ref{tab_invcomp}, the two permutations $\rho_2, \rho_5$ are not compatible with any permutations in $IS_4$. For each permutation $\rho_i\sim I_4$, either $\rho_i$ or its reverse $\widetilde{\rho_i}$ can be included in a compatible set.
	Therefore, there are in total 32 compatible sets for $n=4$, which are listed in Table \ref{tab_CS_4}. They have the maximum size due to \cli the upper bound in \eqref{Eq_PatersonBound}. \cl Note that all these 
	maximal compatible sets give the same spreading matrix $\bPhi$ defined in \eqref{Eq_SpreadingMatrix} (under a permutation of $\bPhi_l$ in $\bPhi$) since for each permutation $\rho \in IS_4$ taken from a compatible set, 
	the spreading sequences $\ba_{\rho}^j$ and $\ba_{\widetilde{\rho}}^j$ are identical due to the fact that $Q_{\rho}(x)=Q_{\widetilde{\rho}}(x)$.
	\begin{table}[t!]
		\centering
		\setlength{\tabcolsep}{3pt} % Adjust column space
		\centering
		\begin{minipage}{0.8\textwidth}
			\centering
			\begin{tabular}{|c|cccccccccccc|}
				\hline
				& 1 & 2 & 3 & 4 & 5 & 6 & 7 & 8 & 9 & 10 & 11 & 12 \\
				\hline
				1 & 0 & - & 6 & 10 & - & 9 & 1 & 4 & 3 & 8 & - & 11 \\
				2 & - & 0 & - & - & - & - & - & - & - & - & - & - \\
				3 & 11 & - & 0 & 7 & - & 12 & 4 & - & 10 & 9 & 1 & 6 \\
				4 & 12 & - & 1 & 0 & - & - & 8 & 11 & 4 & 6 & 7 & 3 \\
				5 & - & - & - & - & 0 & - & - & - & - & - & - & - \\
				6 & 4 & - & 10 & - & - & 0 & 3 & 9 & 12 & 7 & 6 & 8 \\
				7 & 7 & - & 9 & 8 & - & 3 & 0 & 10 & 6 & 4 & 12 & - \\
				8 & 9 & - & - & 6 & - & 4 & 12 & 0 & 1 & 11 & 10 & 7 \\
				9 & 3 & - & 12 & 9 & - & 10 & 11 & 7 & 0 & - & 8 & 1 \\
				10 & 8 & - & 4 & 11 & - & 1 & 9 & 6 & - & 0 & 3 & 10 \\
				11 & - & - & 7 & 1 & - & 11 & 10 & 12 & 8 & 3 & 0 & 9 \\
				12 & 6 & - & 11 & 3 & - & 8 & - & 1 & 7 & 12 & 4 & 0 \\
				\hline
			\end{tabular}
			%        \subcaption{$\rho_i^{-1}\circ \rho_j$}
			%        \label{tab_invcomp_b}
		\end{minipage}
		\caption{Compositions $\rho_i^{-1}\circ \rho_j$ for $1\leq i,j\leq 12$}
		\label{tab_invcomp}
	\end{table}
	
	\begin{table}[t]

        \begin{tabular}{|l|l|l|}
\hline
\( \{I_4,\rho_1,\rho_3,\rho_6,\rho_7,\rho_9\} \) &
\( \{I_4,\rho_1,\rho_3,\rho_6,\rho_7,\rho_{10}\} \) &
\( \{I_4,\rho_4,\rho_7,\rho_8,\rho_{10},\rho_{11}\} \) \\
\hline
\( \{I_4,\rho_3,\rho_4,\rho_7,\rho_{10},\rho_{11}\} \) &
\( \{I_4,\rho_6,\rho_8,\rho_9,\rho_{11},\rho_{12}\} \) &
\( \{I_4,\rho_3,\rho_6,\rho_9,\rho_{11},\rho_{12}\} \) \\
\hline
\( \{I_4,\rho_1,\rho_6,\rho_8,\rho_{10},\rho_{12}\} \) &
\( \{I_4,\rho_1,\rho_3,\rho_6,\rho_{10},\rho_{12}\} \) &
\( \{I_4,\rho_1,\rho_6,\rho_7,\rho_8,\rho_{10}\} \) \\
\hline
\( \{I_4,\rho_3,\rho_4,\rho_{10},\rho_{11},\rho_{12}\} \) &
\( \{I_4,\rho_4,\rho_8,\rho_{10},\rho_{11},\rho_{12}\} \) &
\( \{I_4,\rho_1,\rho_3,\rho_4,\rho_{10},\rho_{12}\} \) \\
\hline
\( \{I_4,\rho_6,\rho_8,\rho_{10},\rho_{11},\rho_{12}\} \) &
\( \{I_4,\rho_6,\rho_7,\rho_8,\rho_9,\rho_{11}\} \) &
\( \{I_4,\rho_4,\rho_7,\rho_8,\rho_9,\rho_{11}\} \) \\
\hline
\( \{I_4,\rho_3,\rho_6,\rho_7,\rho_9,\rho_{11}\} \) &
\( \{I_4,\rho_1,\rho_4,\rho_8,\rho_9,\rho_{12}\} \) &
\( \{I_4,\rho_1,\rho_3,\rho_4,\rho_9,\rho_{12}\} \) \\
\hline
\( \{I_4,\rho_3,\rho_4,\rho_9,\rho_{11},\rho_{12}\} \) &
\( \{I_4,\rho_3,\rho_4,\rho_7,\rho_9,\rho_{11}\} \) &
\( \{I_4,\rho_3,\rho_6,\rho_{10},\rho_{11},\rho_{12}\} \) \\
\hline
\( \{I_4,\rho_4,\rho_8,\rho_9,\rho_{11},\rho_{12}\} \) &
\( \{I_4,\rho_1,\rho_4,\rho_7,\rho_8,\rho_9\} \) &
\( \{I_4,\rho_1,\rho_3,\rho_4,\rho_7,\rho_9\} \) \\
\hline
\( \{I_4,\rho_1,\rho_3,\rho_4,\rho_7,\rho_{10}\} \) &
\( \{I_4,\rho_6,\rho_7,\rho_8,\rho_{10},\rho_{11}\} \) &
\( \{I_4,\rho_3,\rho_6,\rho_7,\rho_{10},\rho_{11}\} \) \\
\hline
\( \{I_4,\rho_1,\rho_4,\rho_7,\rho_8,\rho_{10}\} \) &
\( \{I_4,\rho_1,\rho_4,\rho_8,\rho_{10},\rho_{12}\} \) &
\( \{I_4,\rho_1,\rho_6,\rho_7,\rho_8,\rho_9\} \) \\
\hline
\( \{I_4,\rho_1,\rho_3,\rho_6,\rho_9,\rho_{12}\} \) &
\( \{I_4,\rho_1,\rho_6,\rho_8,\rho_9,\rho_{12}\} \) & \\
\hline
\end{tabular}
		
		\caption{All the 32 maximal compatible sets for $n=4$.}\label{tab_CS_4}
		
	\end{table}
	\ccn
	
	\cli At the end of this section, we discuss the difficulty of the Main Problem~\ref{prob1}.
	\begin{remark}
    The Main Problem~\ref{prob1} can be reformulated as a clique problem.
		Define a graph $G_n$ such that a vertex in $G_n$ denotes a permutation in $S_n$ and an edge between two vertices in $G_n$ represents that the corresponding two permutations in $S_n$ are compatible. 
        Recall from Lemma~\textup{\ref{lem_properties}} $(iii)$ that $\pi\sim\sigma$ if and only if $I_n\sim(\pi^{-1}\circ\sigma)$.
        Hence the graph $G_n$ is an undirected regular graph with $n!$ vertices (where each vertex has the same number of neighbours).
        In practice, edges in $G_n$ can be created by first calculating $IS_n$ and then drawing an edge between two vertices $\pi$ and $\rho$ if $\pi^{-1}\circ\sigma\in IS_n$ (as done for $n=4$ in this section). In essence,
        the Main Problem~\textup{\ref{prob1}} asks to determine or approximate the size
        of the maximum clique in the regular graph $G_n$ and to find an $L$-clique in $G_n$ with $L$ as large as possible.
        
        Note that determining the size of the maximum clique in an arbitrarily given graph is NP-hard, and that 
        there does not even exist any polynomial-time $N^{1-\epsilon}$-approximation algorithm (which, upon the input of a graph $G$ of $N$ vertices, outputs a clique of size that is always at least the maximum clique size of $G$ divided by $N^{1-\epsilon}$) for any constant $\epsilon$ \cite{Haastad1999}. 
        Similarly for regular graphs, it was shown in \textup{\cite{BHK16}} that there exists no polynomial-time algorithm that approximates the maximum size of cliques in a regular graph within a factor of $N^{1/2-\epsilon}$ for any constant $\epsilon$.
        In the theory of computational complexity, existing results show that determining or even approximating the maximum clique size of a regular graph is NP-hard.
        Computationally, 
        the Bron-Kerbosch algorithm, which is one of the most efficient algorithms that lists all maximal cliques in a graph $G$ of $N$ vertices, has worst-case   running time $O(3^{N/3})$~\textup{\cite{Bron1973}}. 
        
        For the specific regular graph $G_n$ derived from $S_n$, with the rapid growth of $N=n!$ as $n$ increases, we believe that the Main Problem~\textup{\ref{prob1}} is computationally intractable for relatively small integers $n$ larger than $10$.
        For instance, for $n=7$, the regular graph $G$ has in total $5040$ vertices with degree $3857$ for each vertex. \cn
		The authors of \textup{\cite{liu2023}} included search results for $4\leq n\leq 9$ with the Bron-Kerbosch algorithm, where they provided only partial search results for $n=7,8,9$.
  %       \cn
  %       Our
		% Main Problem~\textup{\ref{prob1}} is thus essentially to find the maximum clique of a regular graph consisting of $n!$ vertices.  
		% This problem is known to be NP-complete, even when restricted to arbitrary regular \cli graphs~\textup{\cite[pp. 194--195]{GJ79}} \cn (though, the clique number, that is, the maximal order of a complete subgraph in the graph with $m$ vertices, is known to be one of $1,2,\ldots,\lfloor{\frac{m}{2}}\rfloor,m$, for regular graphs). In fact,  it is also known that the clique is NP-hard to approximate, even for regular graphs~\textup{\cite{BHK16}}.
		% For the regular graph $G_n$ with $n!$ vertices, it appears unlikely that the maximum clique problem for $G_n$ can be resolved efficiently. 
  %       Our exhaustive search on $n=4,5,6$ produces maximum compatible sets in $n$ variables of sizes $6, 13, 9$, respectively. 
		% The exhaustive search for compatible sets becomes infeasible quickly as $n$ increases. For instance, for $n=7$, the regular graph has  $5040$ vertices with degree $3857$ for each vertex. \cn
        
	\end{remark} 
	\cn

	\subsection{Extending compatible pairs from $S_{n}$ to $S_{n+m}$}
	\cl In this subsection we will consider extending a permutation $\pi \in IS_n$ with another permutation $\rho \in IS_m$ to obtain a permutation in $IS_{n+m}$. 
	This sets a good starting point for constructing compatible sets in $(n+m)$ variables.
	We will consider the following extension
	\begin{equation}
		\label{Eq_ext_r}
		\begin{split}
			\pi \rho^R &= [\pi(1),\ldots,\pi(n),\rho(1)+n,\ldots,\rho(m)+n].
			%    \\
			%    \rho^L\pi &= [\rho(1)+n,\dots, \rho(m)+n, \pi(1), \dots, \pi(n)].
		\end{split}
	\end{equation}
	For the above extension, it can be verified 
    that for $\pi, \sigma \in S_n$ and $\rho, \varrho\in S_m$, 
	\begin{equation}\label{Eq_Ext_Rev}
		(\pi \rho^R) \circ (\sigma \varrho^R)  = (\pi\circ \sigma)(\rho \circ \varrho)^R.
		%    \quad (\rho^L\pi) \circ (\varrho^L\sigma )  = (\rho \circ \varrho)^L(\pi\circ \sigma).
	\end{equation} 
    In particular, taking $\sigma = \pi^{-1}$ and $\varrho = \rho^{-1}$, we have
    $(\pi \rho^R) \circ (\pi^{-1}(\rho^{-1})^R) = I_{n+m}$, indicating 
	\begin{equation}\label{Eq_Ext_Inv_Rev}
		\begin{split}
			(\pi \rho^R)^{-1} 
            % &=[\pi^{-1}(1),\dots, \pi^{-1}(n), \rho^{-1}(1)+n, \dots, \rho^{-1}(m)+n]\\
			&= \pi^{-1}(\rho^{-1})^R.
			%     (\rho^L\pi)^{-1} &=[\rho^{-1}(1)+n, \dots, \rho^{-1}(m)+n, \pi^{-1}(1),\dots, \pi^{-1}(n)]
			%     = (\rho^{-1})^L\pi^{-1},
		\end{split}
	\end{equation} 
	The following lemma discusses the compatibility between 
	two of such permutations in $IS_{n+m}$.
	\begin{lemma}\label{lem_mutual_compatible}
		Suppose two permutations $\pi_1\rho_1^R$ and $\pi_2\rho_2^R$ are compatible with $I_{n+m}$. 
		Then 
		$$\pi_1\rho_1^R \sim \pi_2\rho_2^R  \text{ if and only if } (\pi_1^{-1}\circ \pi_2)(\rho_1^{-1}\circ\rho_2)^R \sim I_{n+m}.$$
	\end{lemma}
	\begin{proof}
		The statement follows directly from Lemma \ref{lem_properties} $(iii)$, \eqref{Eq_Ext_Rev} and~\eqref{Eq_Ext_Inv_Rev}.
	\end{proof}
	\cn
	In this paper we are mainly concerned with the case of even dimensions. \cli Note that there are only two permutations in $S_2$, namely $I_2$ and $\widetilde{I_2}$, and $\widetilde{I_2}$ is not compatible with $I_2$. 
    \cl
	For the above extension, then, the cases where $n=2$ or $m=2$ cannot lead to useful compatible permutations in $IS_{n+m}$. 
    We will discuss extensions for $n,m\geq 4$. 
	
	\cl Recall that in Definition \ref{Def_WHC} we introduced the Walsh-Hadamard condition (WHC) for quadratic bent functions.  Below we discuss when WHC can be satisfied, which will be frequently used.
	
	Let $n$ be a positive integer, and $i,j$ be two integers with $1\leq i <j \leq n$. \cli Let $u,\,v \in \F_2$ and define a flat  
	\(\Omega_{u,v} = \{(x_1,\dots,x_n)\in \F_2^n\,:\, x_i = u, x_j = v\}.\)
	It is clear that $\F_2^n$ can be partitioned as $\F_2^n = \Omega_{0,0}\sqcup \Omega_{0,1}\sqcup \Omega_{1,0}\sqcup \Omega_{1,1}$.  
	For a quadratic function $Q(x)$, we denote the $(n-2)$-variable Boolean function $Q\arrowvert_{x_i=u,x_j=v}(x)$
    (which is obtained by restricting 
	$Q(x)$ on the flat $\Omega_{u,v}$) for short as $\mQ\vert_{u,v}(\bar{x})$,
    where $\bar{x}\in \F_2^{n-2}$ denotes the vector obtained by removing $x_i, \,x_j$ from $x$ in $\F_2^n$.  \cl
	For instance, assuming $n=4$, $(i,j)=(1,3)$ and $Q(x) = x_1x_2 + x_1x_4 + x_2x_4 + x_3x_4$,
	we obtain $\mQ\vert_{0,0}(\bar{x}) = Q\vert_{x_1=0,x_3=0}(x) = x_2x_4$, and similarly  $\mQ\vert_{0,1}(\bar{x}) = x_2x_4 + x_4$, $\mQ\vert_{1,0}(\bar{x}) = x_2x_4 + x_2 + x_4$ and $\mQ\vert_{1,1}(\bar{x}) = x_2x_4 + x_2 $, where $\bar{x}=(x_2,x_4)$. Note that for $n=4$, the restriction functions for certain quadratic functions $Q(x)$ can have algebraic degree one; e.g., for
	$Q(x) = x_1x_2 + x_1x_3 + x_1x_4 + x_3x_4$ and $(i,j)=(1,3)$, all the restriction functions $\mQ\vert_{u,v}(\bar{x})$ have algebraic degree one.
	\cli For fixed integers $i,j$ with $1\leq i <j \leq n$,
    the functions $\mQ\vert_{0,0}(\bar{x}), \mQ\vert_{0,1}(\bar{x}), \mQ\vert_{1,0}(\bar{x}), \mQ\vert_{1,1}(\bar{x})$ are decompositions of $Q(x)$ as discussed in \cite[Sec.~V]{CCCF01}. \cl 
	
	For an even integer $n$, the following proposition discusses the relation 
    between the Walsh-Hadamard transforms of \cli a bent function \cl $Q(x)$ and its decompositions $\mQ\vert_{u,v}(\bar{x})$ for $u,v\in \F_2$. 
    % \cli Note that a similar result for near-bent functions can be proven for $n$ odd. \cl
	
\begin{prop} 
    \label{Prop_WHC} \cli Let $n\geq 4$ be an even integer and $i,j$ be integers with $1\leq i <j \leq n$. \cl 
Assume $Q(x)$ is a quadratic bent function on $\F_2^n$, and $\mQ\vert_{u,v}(\bar{x}) = Q\vert_{x_i=u, x_j=v}(x)$, where  $u,v\in \F_2$,
are the decompositions of $Q(x)$. Then for any $c\in \F_2^n$, 
the Walsh-Hadamard transforms $W_{Q}(c)$ and $W_{\mQ\vert_{u,v}}(\bar{c})$, where $\bar{c}$ is obtained by removing $c_i,c_j$ from $c\in \F_2^n$, satisfy one of the following:
\begin{enumerate}[label=(\roman*)] 
\item three of $\vert W_{\mQ\vert_{u,v}}(\bar{c})\vert$ are zero and the remaining one equals $2^{\frac{n}{2}}$; this is equivalent to 
$\prod_{\alpha,\beta\in \F_2}W_{Q}(c+\alpha e_i + \beta e_j)=2^{2n}$, the WHC on $(i,j)$; or
\item \cli $\vert W_{\mQ\vert_{u,v}}(\bar{c})\vert=2^{\frac{n-2}{2}}$ for all $u,v\in \F_2$ \cl with three of them having the same sign and the remaining one having the opposite sign; this is equivalent to $\prod_{\alpha,\beta\in \F_2}W_{Q}(c+\alpha e_i + \beta e_j)=-2^{2n}$,
		\end{enumerate}    
        where $e_i, e_j$ are as given in Definition \ref{Def_WHC}.
	\end{prop}	
	\begin{proof}
		According to the definition of $\mQ\vert_{u,v}(\bar{x})$, the 
		Walsh-Hadamard transform of $Q(x)$ at $c\in \F_2^n$ satisfies
		\begin{equation}\label{Eq_WH_Q}
		\begin{aligned}
		W_{Q}(c) & = \sum_{x\in \F_2^n}(-1)^{Q(x)+L_c(x)} 
		\\ &=
		\sum_{x_i,x_j\in \F_2}\sum_{\bar{x}\in \F_2^{n-2}}(-1)^{Q(x)+c_ix_i+c_jx_j+L_{\bar{c}}(\bar{x})} 
		\\ &=
		\sum_{u,v\in \F_2}\sum_{\bar{x}\in \F_2^{n-2}}(-1)^{\mQ\vert_{u,v}(\bar{x})+c_iu+c_jv+L_{\bar{c}}(\bar{x})} 
		\\ &= \sum_{u,v\in \F_2}(-1)^{c_iu+c_jv}W_{\mQ\vert_{u,v}}(\bar{c})
		\\&=W_{\mQ\vert_{0,0}}(\bar{c}) + (-1)^{c_i}W_{\mQ\vert_{1,0}}(\bar{c}) + (-1)^{c_j}W_{\mQ\vert_{0,1}}(\bar{c}) + (-1)^{c_i+c_j}W_{\mQ\vert_{1,1}}(\bar{c}) 
		\\&= \pm 2^{\frac{n}{2}}.
		\end{aligned}
		\end{equation}
		In a similar manner, 
		for $\alpha,\beta\in \F_2$, we have 
		\begin{equation*}
		\begin{aligned}
		W_{Q}(c+\alpha e_i + \beta e_j) & 
		=  \sum_{x\in \F_2^n}(-1)^{Q(x)+(c_i+\alpha )x_i+(c_j+\beta)x_j+L_{\bar{c}}(\bar{x})} 
		%				\\			&=	\sum_{u,v\in \F_2}\sum_{\bar{x}\in \F_2^{n-2}}(-1)^{Q(x)+(c_i+\alpha )x_i+(c_j+\beta)x_j+L_{\bar{c}}(\bar{x})} 
		\\&
		= \sum_{u,v\in \F_2}(-1)^{(c_i+\alpha )u+(c_j+\beta)v}W_{\mQ\vert_{u,v}}(\bar{c}),
		\end{aligned}
		\end{equation*} which gives 
		\begin{equation}\label{Eq_WH_Q_uv}
		\left[\begin{array}{c}
		W_{Q}(c) \\
		W_{Q}(c+e_i) \\
		W_{Q}(c+e_j) \\
		W_{Q}(c+e_i+e_j) 
		\end{array}\right]		
		= 
		\cli\begin{bmatrix}
		1 & 1 & 1 & 1 \\
		1 & -1 & 1 & -1 \\
		1 & 1 & -1 & -1 \\
		1 & -1 & -1 & 1 \\
		\end{bmatrix}\cn
		\left[\begin{array}{r}
		W_{\mQ\vert_{0,0}}(\bar{c})  \\
		(-1)^{c_i}W_{\mQ\vert_{1,0}}(\bar{c}) \\
		(-1)^{c_j}W_{\mQ\vert_{0,1}}(\bar{c}) \\
		(-1)^{c_i+c_j}W_{\mQ\vert_{1,1}}(\bar{c})
		\end{array}	\right].
		\end{equation}			
		
		Note that when $n=4$, the quadratic function $\mQ\vert_{u,v}(\bar{x})$ for all 
		$u,v\in \F_2$ can have algebraic degree one or two, and when $n>4$, 
		the functions $\mQ\vert_{u,v}(\bar{x})$ for all 
		$u,v\in \F_2$ are quadratic. We start with the case where $n=4$ and 
		$\mQ\vert_{u,v}(\bar{x})$ has algebraic degree one for each $u,v\in \F_2$.%, and then discuss the case where $\mQ\vert_{u,v}(\bar{x})$ are all quadratic.
		
		When $n=4$ and $\mQ\vert_{u,v}(\bar{x})$ has algebraic degree one, it is well known that 
		the Walsh-Hadamard transforms $W_{\mQ\vert_{u,v}}(\bar{c})$ equals $2^{n-2}$ or $-2^{n-2}$ at one point $\bar{c}$ and 
		equals zero at all the remaining points $\bar{c} \in \F_2^2$. In addition, it follows from \eqref{Eq_WH_Q}
		that for any $c\in \F_2^n$,
		\begin{equation*}\label{Eq_Q-2}
		W_{\mQ\vert_{0,0}}(\bar{c}) + (-1)^{c_i}W_{\mQ\vert_{1,0}}(\bar{c}) + (-1)^{c_j}W_{\mQ\vert_{0,1}}(\bar{c}) + (-1)^{c_i+c_j}W_{\mQ\vert_{1,1}}(\bar{c}) = \pm 2^{\frac{n}{2}}.
		\end{equation*}
		For $n=4$, one has $n-2 = n/2 =2$ and four choices of $\bar{c}$ in $\F_2^{n-2}$.
		Thus, for each $\bar{c}\in \F_2^2$, 
		three of $W_{\mQ\vert_{0,0}}(\bar{c}), W_{\mQ\vert_{1,0}}(\bar{c}), W_{\mQ\vert_{0,1}}(\bar{c}), W_{\mQ\vert_{1,1}}(\bar{c})$
		are equal to zero and the remaining one is equal to $\pm 2^{\frac{n}{2}}$. 
  %       Furthermore, 
		% given a point $c\in \F_2^n$, suppose $W_{Q}(c)=(-1)^{c_iu+c_jv}W_{\mQ\vert_{u,v}}(\bar{c})=\pm 2^{\frac{n}{2}}$ for certain $(u,v)\in \F_2^2$. 
        Without loss of generality, 
		%we take $(u,v)=(0,1)$ as an example. 
		%That is to say,
        for a point $c\in \F_2^n$, we assume $W_{\mQ\vert_{0,0}}(\bar{c})=W_{\mQ\vert_{1,0}}(\bar{c})=W_{\mQ\vert_{1,1}}(\bar{c})=0$, and $W_{Q}(c) = (-1)^{c_j}W_{\mQ\vert_{0,1}}(\bar{c})= \pm 2^{\frac{n}{2}}$.
		Then, according to \eqref{Eq_WH_Q_uv},  for the given point $c$,  it follows that
		\[
		\left[\begin{array}{c}
		W_{Q}(c) \\
		W_{Q}(c+e_i) \\
		W_{Q}(c+e_j) \\
		W_{Q}(c+e_i+e_j) 
		\end{array}\right]	 = 
        \begin{bmatrix}
		1 & 1 & 1 & 1 \\
		1 & -1 & 1 & -1 \\
		1 & 1 & -1 & -1 \\
		1 & -1 & -1 & 1 \\
		\end{bmatrix}
		\left[\begin{array}{r}
		0  \\
		0 \\
		(-1)^{c_j}W_{\mQ\vert_{0,1}}(\bar{c}) \\
		0
		\end{array}	\right]=
		\begin{bmatrix}
		  1 \\ 1 \\ -1  \\ -1
		\end{bmatrix}\cdot (-1)^{c_j}W_{\mQ\vert_{0,1}}(\bar{c}),
		\] which implies
		\[
		\begin{aligned}
		\prod_{\alpha, \beta\in \F_2}W_{Q}(c+\alpha e_i+\beta e_j) 
		= ((-1)^{c_j}W_{\mQ\vert_{0,1}(\bar{c})})^4 = 2^{2n}.
		\end{aligned}
		\]		
		
		When $n\geq 4$ and $\mQ\vert_{u,v}(\bar{x})$ is a quadratic function on $\F_2^{n-2}$ for any $u,v\in\F_2$, it is easily seen that $\mQ\vert_{u,v}(\bar{x})$ for all $u,v\in \F_2^n$ have the same quadratic terms.
		Let $r$ be the rank of $\mQ\vert_{u,v}(\bar{x})$. Note that $r$ is even since it is equal to the rank of the symplectic matrix of $\mQ\vert_{u,v}(\bar{x})$ that is skew symmetric.       
		It follows from \eqref{Eq_quad} that $W_{\mQ\vert_{u,v}}(\bar{c}) \in \{0, \pm 2^{(n-2)-\frac{r}{2}}\}$.
		Furthermore,  the derivative $Q(x+a) + Q(x)$ for any nonzero $a\in \F_2^n$ is linear, thus balanced. 
		According to~\cite[Theorem~V.4, Equality~(33)]{CCCF01}, we have
		\begin{equation}\label{Eq_G2}
		W_{\mQ\vert_{0,0}}^2(\bar{c})  + W_{\mQ\vert_{0,1}}^2(\bar{c}) + W_{\mQ\vert_{1,0}}^2(\bar{c}) + W_{\mQ\vert_{1,1}}^2(\bar{c}) = 2^n.
		\end{equation}
		This together with \eqref{Eq_WH_Q} implies that the rank $r$ can be either $n-4$ or $n-2$. Moreover, 
		when $r=n-4$, three of $W_{\mQ\vert_{u,v}}(\bar{c})$ for $u,\,v \in \F_2$ are zero and the remaining one equals $\pm  2^{\frac{n}{2}}$;
		when $r=n-2$, for $u,\,v \in \F_2$ all $(-1)^{c_iu+c_jv}W_{\mQ\vert_{u,v}}(\bar{c})=\pm  2^{\frac{n-2}{2}}$, where
		three of them have the same sign and the remaining one has the opposite sign (otherwise it contradicts \eqref{Eq_WH_Q}).
		
		When $r=n-4$, similarly to the discussion for the case where  $\mQ\vert_{u,v}(\bar{x})$ has algebraic degree one for $n=4$, 
		one can assume that given $c\in \F_2^n$, $W_{\mQ\vert_{u,v}(\bar{c})}=\pm 2^{\frac{n}{2}}$ for a certain pair $(u,v)\in \F_2^2$, we can then 
		apply \eqref{Eq_WH_Q_uv} to obtain 
		\[
		\begin{aligned}
		\prod_{\alpha, \beta\in \F_2}W_{Q}(c+\alpha e_i+\beta e_j) 
		= ((-1)^{c_j}W_{\mQ\vert_{u,v}(\bar{c})})^4 = 2^{2n}.
		\end{aligned}
		\] This proves the first claim and implication in Case $(i)$.
		
		When $r=n-2$, again by \eqref{Eq_WH_Q_uv}, we have 
		\[
		\sum_{\alpha,\beta\in \F_2}W_{Q}(c+\alpha e_i+\beta e_j) = 4 W_{\mQ\vert_{0,0}}(\bar{c}) = \pm 2^{\frac{n+2}{2}}.
		\]
		Since $W_{Q}(c+\alpha e_i+\beta e_j) =\pm 2^{\frac{n}{2}}$ for all $\alpha, \beta \in \F_2$, it follows that 
		three of them have the same sign and the remaining one has the opposite sign, which implies 
		\[
		\prod_{\alpha, \beta \in \F_2}W_{Q}(c+\alpha e_i+\beta e_j) 
		= ((-1)^{c_j}W_{\mQ\vert_{u,v}(\bar{c})})^4 = -2^{2n}.
		\]
		This proves the first claim and implication in Case $(ii)$. 

        The equivalence follows from the fact that only the first claim in i) and ii) are possible; therefore, the Walsh-Hadamard condition implies the first claim in $(i)$, and the corresponding negated condition implies the first claim in $(ii)$. 
	\end{proof}
		As shown in Proposition \ref{Prop_WHC}, a quadratic bent function $Q(x)$ satisfies the WHC on $(i,j)$, where $1\leq i< j\leq n$, if and only if the decomposition functions $\mQ\vert_{u,v}(\bar{y})$
	satisfy the properties in Case $(i)$.
	
	Now we are ready to present the first main theorem below.
	\begin{theorem} \label{Th_F_G}
		Let $n, m\geq 4$ be even, and let $\pi\in IS_n$, $\rho\in IS_m$ and define $f(x)=Q_{\pi}(x)+Q_{I_n}(x), \, g(y)=Q_{\rho}(y)+Q_{I_m}(y)$, respectively.
		Then the permutation $\pi\rho^R$ is compatible with $I_{n+m}$ if and only if one of the following conditions holds:
		\begin{enumerate}[label=(\roman*)]
			\item $(\pi(n)-n)(\rho(1)-1)=0$; or 
			% 	\item $(\rho(1)-1)(\pi(n)-n)\neq 0$ and for any $(a,b)\in \F_2^n$, 
			% 	$$
			% \prod_{i,j\in \F_2} W_{g_{i,j}}(\bar{b}) W_f(a + ie_{\pi(n)} + je_{n}) = 0, \, -2^{2n+2m-4},
			% 	$$where $\bar{b}$ is the vector in $\F_2^{m-2}$ obtained by removing $b_1,\,b_{\rho(1)}$ from $b\in \F_2^m$, and
			% 	$e_k$ with $1\leq k\leq n$ is the $k$-th unit vector in $\F_2^n$.
			\item $(\pi(n)-n)(\rho(1)-1)\neq 0$ and at least one of $f(x)$ and $g(y)$ satisfy the WHC on $(\pi(n), n)$, $(1,\rho(1))$, respectively.

%			when for certain $b\in \F_2^m$,  all $W_{g_{i,j}}(\bar{b})$ for $i,\,j\in \F_2$ are equal to $\pm 2^{\frac{m-2}{2}}$, the functions $g_{i,j}(\bar{y})$ and $f(x)$ 
%			\[
%			\prod_{i,j\in \F_2} W_{g_{i,j}}(\bar{b}) = -2^{2(m-2)} \text{ and } \prod_{i,j\in \F_2} W_f(a + ie_{\pi(n)} + je_{n})  = 2^{2n}, \forall \, a\in \F_2^n.
%			\]
%			where $\bar{b}$ is the vector in $\F_2^{m-2}$ obtained by removing $b_1,\,b_{\rho(1)}$ from $b\in \F_2^m$, and
%			$e_k$ with $1\leq k\leq n$ is the $k$-th unit vector in $\F_2^n$.
			% 	\item $(\rho(1)-1)(\pi(n)-n)\neq 0$  and \[
			% 	\begin{split}
			% 	\prod_{i,j\in \F_2}(-1)^{a_{\pi(n)}i + a_n j}W_{f_{i,j}}(\bar{a})W_g(b + i \epsilon_{1}+j\epsilon_{\rho(1)})
			% 	=0,\, -2^{2n+2m-4},
			% 	\end{split}
			% 	\]where $\bar{a}$ is the vector in $\F_2^{n-2}$ obtained by removing $a_{\pi(n)},\,a_{n}$ from $a\in \F_2^n$, and
			% 	$\epsilon_k$ is the $k$-th unit vector in $\F_2^m$.
		\end{enumerate}  
	\end{theorem} 
    \cn
	
	\begin{proof}
		\ccb 
		For the permutation $\pi\rho^R \in S_{n+m}$, we denote a Boolean function
		$h(x,y)=Q_{\pi\rho^R}(x,y)+Q_{I_{n+m}}(x,y)$ on $\mathbb{F}_2^n\times \mathbb{F}_2^m$. 
		From the definition of $\pi\rho^R$, it follows directly that
		\ccb
		\begin{equation}\label{eq_h}
			\begin{split}
				h(x,y)
				% &=Q_{\pi\rho^R}(x,y)+Q_{I_{n+m}}(x,y)
				&=
				\sum_{i=1}^{n-1}(x_{\pi(i)}x_{\pi(i+1)} + x_ix_{i+1})
				+ x_{\pi(n)}y_{\rho(1)} + x_ny_1\\
				&\qquad \qquad\qquad    + \sum_{j=1}^{m-1}(y_{\rho(j)}y_{\rho(j+1)} + y_jy_{j+1})
				\\&=f(x)+g(y)+(x_{\pi(n)}y_{\rho(1)}+x_ny_1).
			\end{split}
		\end{equation}
		\ccn
		The Walsh-Hadamard transform of $h(x,y)$ at a point $(a,b)\in \F_2^n \times \F_2^m$ is given by 
		\begin{equation}\label{Eq_WH_h}
			W_h(a,b)=\sum_{(x,y)\in \F_2^{n+m}} (-1)^{f(x)+g(y)+\ccb L_a(x) +L_b(y)\ccn + (x_{\pi(n)}y_{\rho(1)}+x_ny_1)}.
		\end{equation}
		Below we shall investigate the condition such that \ccb $\vert W_h(a,b)\vert=2^{\frac{n+m}{2}}$,  for all $(a,b)\in \F_2^{n\times m}$ according to the values of $\rho(1)$ and $\pi(n)$.\ccn
		
		\medskip 
		
		For Case $(i)$, \cl we consider three subcases, namely, $\pi(n) = n$ and $\rho(1)=1$;  $\pi(n)\neq n$ and $\rho(1)=1$;  $\pi(n)= n$ and $\rho(1)\neq 1$. \cn
		
\smallskip

{\em Subcase $1$. $\pi(n) = n$ and $\rho(1)=1$:}  %      When $\pi(n) = n$ and $\rho(1)=1$, 
In this subcase, we have $h(x,y)=f(x)+g(y)$. It follows from~\eqref{Eq_WH_h} that
		\[
		W_h(a,b) =  \sum_{(x,y)\in \F_2^{n+m}} (-1)^{f(x)+g(y)+\ccb L_a(x)+L_b(y)\ccn }
		= W_f(a)W_g(b) \in \{\pm 2^{\frac{n+m}{2}}\},
		\]
		which implies $h(x,y)$ is bent, and then $\pi\rho^R\sim I_{n+m}$.

\smallskip 

{\em Subcase $2$. $\pi(n)\neq n$ and $\rho(1)=1$:} %When $\pi(n)\neq n$ and $\rho(1)=1$,
In this subcase, we have $h(x,y)=f(x)+g(y)+y_1(x_{\pi(n)}+x_n)$.
		Then, 
		\begin{equation*}
			\begin{split}
			&W_h(a,b) \\
				=&\sum_{x\in\F_2^n} (-1)^{f(x)+L_a(x)}\left[ \sum_{\substack{y\in \F_2^{m}\\ y_1=0}}(-1)^{g(y)+L_b(y)}+ (-1)^{x_n + x_{\pi(n)}} \sum_{\substack{y\in \F_2^{m}\\y_1=1}} (-1)^{g(y)+L_b(y)} \right].
%				=&W_f(a)G_0+W_f(a+e_n+e_{\pi(n)})G_1,
			\end{split}
		\end{equation*} \cl
		For simplicity, we denote
		\[G_0(b)=\sum_{\substack{y\in \F_2^{m}\\ y_1=0}} (-1)^{g(y)+L_b(y)}, \quad 
		G_1(b) = \sum_{\substack{y\in \F_2^{m}\\ y_1=1}}(-1)^{g(y)+L_b(y)}.
		\] \cn 
		Then we have 
			\begin{equation}\label{Eq-G-case-1}
		W_h(a,b)  = W_f(a)G_0(b)+W_f(a+e_n+e_{\pi(n)})G_1(b)
		\end{equation} and 
		$$
		G_0(b)+G_1(b) =W_g(b)=\pm 2^{\frac{m}{2}}.
		$$ 
		Since the Boolean functions \ccb $g(y)+L_b(y)$   for $y_1=0$ and $y_1=1$ have the same quadratic terms, it follows from \eqref{Eq_quad} 
		that $G_0(b), G_1(b) \in \{0, \pm 2^{m-1 - \frac{r}{2}}\}$ for an even integer~$r$ with $0\leq r \leq (m-1)$. 
		By the equality $G_0(b)+G_1(b) = \pm 2^{\frac{m}{2}}$, it is clear that  the product of $G_0(b)$ and $G_1(b)$ must be zero, i.e.,
		$(G_0(b),G_1(b))\in \{(0,\pm 2^{\frac{m}{2}}), (\pm 2^{\frac{m}{2}}, 0)\}$,  since if $\vert G_0(b) \vert =\vert G_1(b)\vert = 2^{\frac{m}{2}-1}$, then $r=m$, which is not possible for functions on $m-1$ variables, and the other possibilities sum to zero. \ccn
		Each of these cases implies that $W_{h}(a,b)\in \{\pm 2^{\frac{n+m}{2}}\}$ and then $\pi \rho^R\sim I_{n+m}$.
		
\smallskip

	{\em Subcase $3$. $\pi(n)=n$ and $\rho(1)\neq 1$:}	%When  $\pi(n)=n$ and $\rho(1)\neq 1$, 
	In this subcase,	we have $h(x,y)=f(x)+g(y)+x_n(y_1+y_{\rho(1)})$. As in
		the case where $\pi(n)\neq n,\, \rho(1)=1$, \cli letting $$F_0(a)=\displaystyle\sum_{\substack{x\in \F_2^{n}\\ x_n=0}}(-1)^{f(x)+\ccb L_a(x)\ccn }, \quad 
        F_1(a)=\displaystyle\sum_{\substack{x\in \F_2^{n}\\ x_n=1}}(-1)^{f(x)+\ccb L_a(x)},$$ one has
		\begin{equation*}
			\begin{split}
				W_h(a,b)=W_g(b)F_0(a)+W_g(b+e_1+e_{\rho(1)})F_1(a). 
			\end{split}
		\end{equation*}
		 
		Following similar arguments as in the previous subcase, we have $(F_0(a),F_1(a))\in \{(0,\pm 2^{\frac{n}{2}}), (\pm 2^{\frac{n}{2}}, 0)\}$,
		which implies $\vert W_{h}(a,b)\vert =2^{\frac{n+m}{2}}$ and then $\pi\rho^R\sim I_{n+m}$.
		\ccn
		
		\medskip 
		
		\cl  Case $(ii)$.
		For $(\pi(n)-n)(\rho(1)-1)\neq 0$, \cn we have $h(x,y)=f(x)+g(y)+(x_{\pi(n)} y_{\rho(1)}+x_n y_1)$. 
		In this case, 
		we need to investigate $W_h(a,b)$ by explicit evaluations on $y_1, y_{\rho(1)}$ or on $x_n, x_{\pi(n)}$.
		
		Considering the explicit values of $y_1, y_{\rho(1)}$ in the calculation of $W_g(b)$, we have, \cli for any $u,v\in \F_2$,
		\begin{align*}
			\sum_{\substack{y_{\rho(1)}=u\\y_1=v}} (-1)^{g(y)+\ccb L_b(y)\ccn}
			&=(-1)^{b_{\rho(1)}u +b_{1}v}\sum_{\substack{\bar{y}\in\F_2^{m-2}}} (-1)^{g_{u,v}(\bar{y})+\ccb L_{\bar{b}}(\bar{y})\ccn }\\
			&
			= (-1)^{b_{\rho(1)}u +b_{1}v} W_{g_{u,v}}(\bar{b}),
		\end{align*} 
		where $g_{u,v}(\bar{y})$ is the restriction of $g(y)$ on $\Lambda_{u,v}=\{y\in \F_2^m\,:\, y_1 = u, \, y_{\rho(1)} = v\}$, \cn
		and $\bar{y},\bar{b}$ are the vectors obtained by removing the $1$-st and $\rho(1)$-th entries of $y$ and $b$ in $\F_2^m$, respectively.
	   For simplicity, denote $G_{u,v}(b) = (-1)^{b_{\rho(1)}u +b_{1}v} W_{g_{u,v}}(\bar{b})$ and $\widetilde{F_{u,v}}(a)=W_f(a + ue_{\pi(n)} + ve_{n})$ for $u,v\in \F_2$,
       where $e_{\pi(n)}, e_n$ are the $\pi(n)$-th, respectively, the $n$-th row of the dimension-$n$ identity matrix.
       \cn
		Then 
		\begin{equation}\label{Eq_G_1}
			W_g(b) = G_{0,0}(b) + G_{0,1}(b) + G_{1,0}(b) + G_{1,1}(b) = \pm 2^{\frac{m}{2}}
		\end{equation}
		and
		\begin{equation}
			\label{Eq_WH_GF}
			\begin{split}
				&W_h(a,b) 
			\\	&= \sum_{(x,y)\in \F_2^{n+m}} (-1)^{f(x)+g(y)+\ccb L_a(x)+L_b(y) \ccn+ (x_{\pi(n)}y_{\rho(1)}+x_ny_1)} \\
				&= \sum_{x\in \F_2^n} (-1)^{f(x)+\ccb L_a(x)}\ccn \sum_{y\in \F_2^m}(-1)^{g(y)+\ccb L_b(y) \ccn+ (x_{\pi(n)}y_{\rho(1)}+x_ny_1)} 
				\\
				&=\sum_{x\in \F_2^n} (-1)^{f(x)+\ccb L_a(x)\ccn} \left[G_{0,0} (b)+ (-1)^{x_{n}}G_{0,1}(b) 
				 \right.\\ & \hspace{4.2cm} \left.
				+ (-1)^{x_{\pi(n)}}G_{1,0}(b) + (-1)^{x_{\pi(n)}+x_n}G_{1,1}(b)\right]
				\\
				&= G_{0,0}(b)\widetilde{F_{0,0}}(a)+G_{0,1}(b)\widetilde{F_{0,1}}(a)+G_{1,0}(b)\widetilde{F_{1,0}}(a)+G_{1,1} (b)\widetilde{F_{1,1}}(a).
			\end{split}
		\end{equation} \cl
		Similarly, for $(a,b)\in \F_2^{n+m}$ and $u,v\in \F_2$, define \cli
		\[
		\widetilde{G_{u,v}}(b) = W_{g}(b+u \varepsilon_1 + v\varepsilon_{\rho(1)} ) \text{ and } F_{u,v}(a) = (-1)^{a_{\pi(n)}u +a_{n}v} W_{f_{u,v}}(\bar{a}),
		\] where $\varepsilon_1, \varepsilon_{\rho(1)}$ are the first and $\rho(1)$-th row of the $m\times m$ identity matrix over $\F_2$. \cn 
		By the symmetry of $f(x)$ and $g(y)$ in $h(x,y)$, we have 
		\[
		W_h(a,b)  = F_{0,0}(a)\widetilde{G_{0,0}}(b)+F_{0,1}(a)\widetilde{G_{0,1}}(b)+F_{1,0}(a)\widetilde{G_{1,0}}(b)+F_{1,1} (a)\widetilde{G_{1,1}}(b).
		\]
		Below we investigate the necessary and sufficient condition to have $\vert W_{h}(a,b) \vert = 2^{\frac{n+m}{2}}$, for all $(a,b)\in \F_2^{n+m}$.
		
		According to Proposition \ref{Prop_WHC}, for even $m\geq 4$, the values of $G_{u,v}(b)$ as $b$ ranges through $\F_2^m$ can  be divided into two subcases:
		three of $G_{u,v}(b)$ are zero and the remaining one is $\pm 2^{\frac{m}{2}}$, which is equivalent to the fact that $g(y)$ satisfies the WHC on $(1, \rho(1))$; 
	   all $G_{u,v}(b)$ are equal to $\pm 2^{\frac{m-2}{2}}$ with three of them having the same sign and the fourth one having the opposite sign.
	  
	    For the first subcase, it is easily seen that 
	    \[
			\begin{aligned}
%				W_h(a,b)  &=
				G_{0,0}(b)\widetilde{F_{0,0}}(a)+G_{0,1}(b)\widetilde{F_{0,1}}(a)+G_{1,0}(b)\widetilde{F_{1,0}}(a)+G_{1,1} (b)\widetilde{F_{1,1}}(a) 
				& = \pm 2^{\frac{n+m}{2}}
			\end{aligned}
	    \] since only one term in the sum is nonzero.

		For the second subcase,  for any $u,v\in \F_2$ and $(a,b)\in \F_2^{n+m}$, we have
		\[
			\vert G_{u,v}(b) \vert = 2^{\frac{m-2}{2}} \text{ and } \vert G_{u,v}(b)\widetilde{F_{u,v}}(a) \vert = 2^{\frac{n+m-2}{2}}. 
		\]
		Denote $\eta_{u,v}(a,b)=G_{u,v}(b)\widetilde{F_{u,v}}(a)/2^{\frac{n+m-2}{2}}$ for $u,v\in \F_2$. Then
		$$
		G_{0,0}(b)\widetilde{F_{0,0}}(a)+G_{0,1}(b)\widetilde{F_{0,1}}(a)+G_{1,0}(b)\widetilde{F_{1,0}}(a)+G_{1,1} (b)\widetilde{F_{1,1}}(a)  = \pm 2^{\frac{m+n}{2}}
		$$
		if and only if for any $(a,b)\in \F_2^n$,
		$
					(\eta_{0,0}(a,b), \eta_{0,1}(a,b), \eta_{1,0}(a,b), \eta_{1,1}(a,b)) \in 
			\{\pm (1,1,1,-1),  \pm (1,1,-1,1),\pm (1,-1,1,1),\pm(-1,1,1,1)\}.
		$ That is to say, for any $(a,b)\in \F_2^n$, three of the $\eta_{0,0}(a,b),\eta_{0,1}(a,b),\eta_{1,0}(a,b),\eta_{1,1}(a,b)$ have the same sign, and the fourth one has the opposite sign. Therefore,
		$W_{h}(a,b)=\pm 2^{\frac{m+n}{2}}$ if only if
		\begin{equation}\label{Eq-h-bent}
				\prod_{u,v\in \F_2} G_{u,v}(b) W_f(a+ue_{\pi(n)} + ve_n) = - 2^{2n+2m-4},	
		\end{equation}
		which, by the fact that $\prod_{u,v\in \F_2} G_{u,v}(b)=-2^{2(m-2)}$, is equivalent to 
		$$
		  \prod_{u,v\in \F_2}W_f(a+ue_{\pi(n)} + ve_n) = 2^{2n}, \text{ for any } a\in \F_2^n,
		$$ i.e., $f(x)$ satisfies the WHC on $(1,\rho(1))$.
		
		Combining the above two subcases, we see that in the case of$(\pi(n)-n)(\rho(1)-1)\neq 0$, if one of 
		 $f(x)$ and $g(y)$ satisfies the WHC, then $h(x,y)$ is bent on $\F_2^{n+m}$. 
		 In addition, when none of them satisfies the WHC, according to Proposition \ref{Prop_WHC} $(ii)$, the decompositions of $f(x)$
		 and $g(y)$ are all bent functions. In this case, since $\prod_{u,v\in \F_2} G_{u,v}(b)=-2^{2(m-2)}$ and $\prod_{u,v\in \F_2} F_{u,v}(a)=-2^{2(n-2)}$, 
		 it is easy to verify that $h(x,y)$ cannot be bent.	
		The desired claims in Case~$(ii)$ thus follow.
	\end{proof} \cn

	Below we provide a theorem for $n, m\geq 4$, odd $n$ and even $m$, about the extension. Since the proof is rather similar to the prior one, we suppress it.
	\begin{theorem} \label{Th_F_G_odd}  \cl
		For integers $n,\, m\geq 4$ with odd $n$ and even $m$, let
		$\pi\in IS_n$, $\rho\in IS_m$ and define $f(x)=Q_{\pi}(x)+Q_{I_n}(x)$ and $g(y)=Q_{\rho}(y)+Q_{I_m}(y)$.
		Then $\pi\rho^R$ is compatible with $I_{n+m}$ if one of the following conditions holds:
\begin{enumerate}[label=(\roman*)]
\item $\rho(1)=1$; or
\item $(\rho(1)-1)(\pi(n)-n) \neq 0$ and for any $(a,b)\in \F_2^{n+m}$,
$$
\prod_{u,v\in \F_2} W_{g_{u,v}}(\bar{b})W_f(a + ue_{\pi(n)} + ve_{n}) \cli \in\left\{ 0, -2^{2n+2m-6}\right\},\cn
$$ 
where $\bar{b}$ is the dimension-$(m-2)$ vector derived from $b$ by removing its first and $\rho(1)$-th entries. \cn
			% \item $\pi(n)\neq n$, $\rho(1)\neq 1$ and 
			%     $$
			%         \prod_{u,v\in \F_2}(-1)^{a_{\pi(n)}i + a_n j}W_{f_{i,j}}(\bar{a})W_g(b + i e_{\rho(1)}+je_{1})= - 2^{2n+2m-6}.
			%     $$
			% \item $\pi(n)\neq n$, $\rho(1)\neq 1$,  $m=4$, $g_{i,j}$ is linear; or
			%        \item $\pi(n)\neq n$, $\rho(1)\neq 1$, both $g_{i,j}, f_{i,j}$ are quadratic and one of the following holds for any $(a,b)\in \F_2^{n\times m}$, 
			%          \[
			%          \begin{split}
			%               &\prod_{i,j\in \F_2} (-1)^{b_{\rho(1)}i +b_{1}j} W_{g_{i,j}}(\bar{b})W_f(a + ie_{\pi(n)} + je_{n}) = 0 \text{ or } - 2^{2n+2m-6},
			%              \\
			%              &\prod_{i,j\in \F_2}(-1)^{a_{\pi(n)}i + a_n j}W_{f_{i,j}}(\bar{a})W_g(b + i e_{\rho(1)}+je_{1})= 0 \text{ or } - 2^{2n+2m-6},
			%          \end{split} 
			%          \]
			%    where the vectors $\bar{a}\in \F_2^{n-2},\bar{b}\in \F_2^{m-2}$ are obtained by removing
			%    $a_n, a_{\pi(n)}$ from $a$, and  $b_1,b_{\rho(1)}$
			%    from $b$, respectively.
			% \begin{enumerate}
			% 	\item $W_f(a)=W_f(a+e_n)=W_f(a+e_{\pi(n)})=W_f(a+e_n+e_{\pi(n)})$.
			% 	\item $W_f(a)=W_f(a+e_n)$ and $W_f(a+e_{\pi(n)})=W_f(a+e_n+e_{\pi(n)})$.
			% 	\item $W_f(a)=W_f(a+e_{\pi(n)})$ and $W_f(a+e_n)=W_f(a+e_n+e_{\pi(n)})$ 
			% \end{enumerate}
		\end{enumerate}  
	\end{theorem}

	\ccn 
	
	%Theorems \ref{Th_F_G} and \ref{Th_F_G_odd} characterize the condition for 
	%$\pi\rho^R$ to be compatible with $I_{n+m}$.
	
	%Theorem \ref{Th_F_G} and Lemma \ref{lem_mutual_compatible}
	%provide a guideline to the construction a compatible set of permutations in $S_{n+m}$.
	In the sequel we shall consider extending a permutation $\pi\in IS_n$ with a permutation $\rho$ taken from $IS_4 = \{\rho_1,\rho_2, \ldots, \rho_{12}\}$, which makes the conditions in Theorems \ref{Th_F_G} and \ref{Th_F_G_odd} more explicit to deal with.
	
	\begin{corollary}\label{101024th1}
		Let $\pi$ be a permutation in $IS_{n}$. Then the permutation $\pi\rho^R$ is compatible with $I_{n+4}$
		for any $\rho$ in the set $\{\rho_3,\rho_4,\rho_6, \rho_8,\rho_9, \rho_{11}\} = \{[3,4,1,2], [2,4,3,1],[1,3,4,2],[2,1,4,3],[4,1,3,2],[1,4,2,3]
		\}$.
		%\[
		%\begin{split}
		%&\{\rho_3,\rho_4,\rho_6, \rho_8,\rho_9, \rho_{11}\}\\
		%=&
		%\{[3,4,1,2], [2,4,3,1],[1,3,4,2],[2,1,4,3],[4,1,3,2],[1,4,2,3]
		%\}
		%\end{split}
		%\]
	\end{corollary}
	\begin{proof}
		By Theorems \ref{Th_F_G} and \ref{Th_F_G_odd} Case $(i)$, it is readily seen that 
		$\pi \rho^R \sim I_{n+4}$ for $\rho=\rho_6, \,\rho_{11}$ since $\rho_6(1)=\rho_{11}(1)=1$.
		For the other permutations, we have 
		\allowdisplaybreaks
		\begin{align*}
			Q_{\rho_3}(y)+Q_{I_4}(y) &=  y_2y_3 + y_4y_1,
			\\
			Q_{\rho_4}(y)+Q_{I_4}(y) &=  y_1y_2 + y_2y_3 + y_2y_4 + y_3y_1,
			\\
			Q_{\rho_8}(y)+Q_{I_4}(y) &=  y_2y_3 + y_1y_4,
			\\
			Q_{\rho_9}(y)+Q_{I_4}(y) &=  y_1y_2 + y_3y_4 + y_4y_1 + y_1y_3.
		\end{align*}
		For each permutation $\rho\in \{\rho_3,\rho_4, \rho_8, \rho_9\}$,
		it can be easily verified that 
		the function $g(y) = Q_{\rho}(y)+Q_{I_4}(y)$ restricted on $y_{\rho(1)}=i,y_{1} = j$
		for any $i,j\in \F_2$ is a linear function on $\F_2^2$, implying that the condition in Theorems \ref{Th_F_G} and \ref{Th_F_G_odd} Case $(ii)$ is satisfied.
		It thus follows that $\pi\rho^R \sim I_{n+4}$.
	\end{proof}
	
	%According to Lemma \ref{lem_properties} $(i)$, we immediately have the following corollary.
	%\begin{corollary}\label{cor1}
	%    Let $\pi \in S_n$ be compatible with $I_{n}$. Then the permutation $\rho^L \widetilde{\pi}$ 
	%for any 
	%$\rho \in \{\rho_{1}, \rho_3,\rho_4,\rho_6, \rho_8,\rho_{10}\}
	%$ is compatible with $I_{n+4}$.
	%\end{corollary}
	%\begin{proof}
	%    By Lemma \ref{lem_properties} $(i)$ and the fact
	%$
	%\{\widetilde{\rho_3},\widetilde{\rho_4},\widetilde{\rho_5}, \widetilde{\rho_8},
	%\widetilde{\rho_9}, \widetilde{\rho_{11}}\}
	%$$
	%= 
	%\{
	%\rho_8,\rho_6, \rho_4,\rho_3,\rho_{10}, \rho_{1}
	%\},
	%$ the statement directly follows.
	%\end{proof}

	\medskip 
	We now focus on the case of even $n$, and we study the extensions on the right with the remaining permutations $\rho_1,\rho_2,\rho_5,\rho_7,\rho_{10},\rho_{12}$ in $IS_4$. For this case, 
	the following conditions will be needed.
%	\begin{defn} Given a permutation $\pi\in IS_n$, the quadratic bent function $f(x)=Q_{\pi}(x) + Q_{I_n}(x)$ is said to satisfy the Walsh-Hadamard condition (WHC) if 
%		\begin{equation}\label{whc}
%			W_f(a)W_f(a+e_{n}) W_f(a+e_{\pi(n)})W_f(a+e_{n}+e_{\pi(n)}) = 2^{2n},
%			\text{ for all } a\in \F_2^{n}. 
%		\end{equation} 
%		%    and $f$ is said to satisfy the Negated Walsh-Hadamard condition (NWHC) if
%		%   \begin{equation}\label{nwhc}
%		%W_f(a)W_f(a+e_{n})W_f(a+e_{\pi(n)})W_f(a+e_{n}+e_{\pi(n)})=-2^{2n},\quad \text{ for all } a\in \F_2^{n}.
%		%   \end{equation}
%	\end{defn}
	
	\begin{corollary}\label{121024th2}
		\cl Let $\pi \in IS_n$ with $\pi(n) \neq n$. \cn Then for a permutation $\rho\in \{\rho_1,\rho_2,\rho_5,\rho_7,\rho_{10},\rho_{12}\}$,
		the permutation $\pi\rho^R$ is compatible with $I_{n+4}$ \cl if and only if $f(x)=Q_\pi(x) + Q_{I_{n}}(x)$ fulfills the WHC on $(\pi(n), n)$. \cn
	\end{corollary}
	\begin{proof}
		Note that for $\rho\in \{\rho_1,\rho_2,\rho_5,\rho_7,\rho_{10},\rho_{12}\}$, the quadratic functions $Q_{\rho}(y) + Q_{I_4}(y)$ are given as follows:  \cl
		\begin{align*}
			Q_{\rho_1}(y)+Q_{I_4}(y) &=  y_1y_2 + y_1y_4 + y_2y_4 + y_3y_4,
			\\
			Q_{\rho_2}(y)+Q_{I_4}(y) &=  y_1y_2 + y_1y_3 + y_1y_4 + y_2y_3 + y_2y_4 + y_3y_4,
			\\
			Q_{\rho_5}(y)+Q_{I_4}(y) &=  y_1y_2 + y_1y_3 + y_1y_4 + y_2y_3 + y_2y_4 + y_3y_4,
			\\
			Q_{\rho_7}(y)+Q_{I_4}(y) &=  y_1y_3 + y_2y_3 + y_2y_4 + y_3y_4,
			\\
			Q_{\rho_{10}}(y)+Q_{I_4}(y) &=  y_1y_2 + y_1y_3 + y_1y_4 + y_3y_4,
			\\
			Q_{\rho_{12}}(y)+Q_{I_4}(y) &=  y_1y_3 + y_2y_3 + y_2y_4 + y_3y_4.
		\end{align*} \cn
		The calculations for all $\rho$ in the set are similar, and we will take $\rho=\rho_1=[3,2,4,1]$ as an instance. In this case, letting $\pi'=\pi\rho^R $, $f(x)=Q_{I_n}(x)+Q_\pi(x)$, and $g(y)=Q_{\rho_1}(y)+Q_{I_4}(y)$, we have
		\begin{align*}h(x,y)&=Q_{I_{n+4}}(x,y)+Q_{\pi'}(x,y)=f(x_1,\ldots,x_{n})+g(y_1,y_2,y_3,y_4)+x_{\pi(n)}y_3+x_ny_1\\&=f(x_1,\dots, x_n)+y_1y_2 + y_1y_4 + y_2y_4 + y_3y_4+x_{\pi(n)}y_{3}+x_{n}y_{1}.\end{align*}
		For $g(y) = y_{1}y_{2}+y_{3}y_{4}+y_{2}y_{4}+y_{4}y_{1}$, the function $g_{i,j}(\bar{y})=g\arrowvert_{y_3=i,y_1=j} =y_2y_4 + j(y_2+y_4) + i y_4 $ contains the quadratic term $y_2y_4$ for any $(i,j)\in\F_2^2$. 
		%		Therefore, if $f_{i,j}=f\arrowvert_{x_{\pi(n)}=i,x_n=j}$ is linear for all $(i,j)$ and $n=4$, then from Theorem \ref{Th_F_G} $(ii)$, $\pi'\sim I_{n+4}$. Assume that $f_{i,j}$ contains quadratic term(s) for all $(i,j)$. 
		\cl	Since $\pi(n) \neq n$, the permutation $\pi'$ is compatible with $I_{n+4}$ if and only if the condition in Theorem~\ref{Th_F_G}~$(ii)$ is satisfied. 
		Below we investigate the explicit value of the product in Theorem~\ref{Th_F_G}~$(ii)$.
		\cn
		
		Denote 
		$$W_{h_y}(a,b)=\sum_{x\in \F_2^n}(-1)^{h(x,y)+\ccb L_a(x)+L_b(y)\ccn } = (-1)^{\ccb L_b(y)\ccn } \sum_{x\in \F_2^n}(-1)^{h(x,y)+\ccb L_a(x)\ccn}.
		$$ Table \ref{tab2171024} lists the values of $W_{h_y}$, for all $y\in (y_1,y_2,y_3,y_4)\in \F_2^4$, where $\alpha=(-1)^{b_1}$, $\beta=(-1)^{b_2}$, $\gamma=(-1)^{b_3}$, and $\delta=(-1)^{b_4}$.

		\begin{table}
			\begin{center}
				\begin{tabular}{cccc|l}
					$y_1$&$y_2$&$y_3$&$y_4$&$W_{h_{y}}(a,b)$ \\
					\hline
					$0$&$0$&$0$&$0$&$W_f(a)$\\
					$0$&$0$&$0$&$1$&$W_f(a)\delta$\\
					$0$&$0$&$1$&$0$&$W_f(a+e_{\pi(n)})\gamma$\\
					$0$&$0$&$1$&$1$&$W_f(a+e_{\pi(n)})\gamma\delta(-1)$\\
					$0$&$1$&$0$&$0$&$W_f(a)\beta$\\
					$0$&$1$&$0$&$1$&$W_f(a)\beta\delta(-1)$\\
					$0$&$1$&$1$&$0$&$W_f(a+e_{\pi(n)})\beta\gamma$\\
					$0$&$1$&$1$&$1$&$W_f(a+e_{\pi(n)})\beta\gamma\delta$\\
					$1$&$0$&$0$&$0$&$W_f(a+e_{n})\alpha$\\
					$1$&$0$&$0$&$1$&$W_f(a+e_{n})\alpha\delta(-1)$\\
					$1$&$0$&$1$&$0$&$W_f(a+e_{n}+e_{\pi(n)})\alpha\gamma$\\
					$1$&$0$&$1$&$1$&$W_f(a+e_{n}+e_{\pi(n)})\alpha\gamma\delta$\\
					$1$&$1$&$0$&$0$&$W_f(a+e_{n})\alpha\beta(-1)$\\
					$1$&$1$&$0$&$1$&   $W_f(a+e_{n}) (-1)\alpha\beta\delta$ \\
					$1$&$1$&$1$&$0$& $W_f(a+e_{n}+e_{\pi(n)}(-1)\alpha\beta\gamma$\\
					$1$&$1$&$1$&$1$&$ W_f(a+e_{n}+e_{\pi(n)})\alpha\beta\gamma\delta$ 
				\end{tabular}
				\caption{The Walsh-Hadamard transform of $h$}\label{tab2171024}
			\end{center}
		\end{table}
		
		From Table \ref{tab2171024}, the Walsh-Hadamard transform of $h$ at $(a,b)$ can be expressed as follows:
		% \textcolor{red}{{\bf Pante}: There are some typos in the calculation below, or in the table; I'll look it over tomorrow, as I'm tired tonight.}
		%Checked and correct, what is the problem?
		\begin{align*}
			W_h(a,b)=&\sum_{y\in\F_2^4}W_{h_y}(a,b)\\
			=&W_f(a)(1+\delta+\beta(1-\delta))+W_f(a+e_{\pi(n)})\gamma(1-\delta+\beta(1+\delta))\\
			&\,+W_f(a+e_{{n}})\alpha(1-\delta-\beta(1+\delta))\\&+W_f(a+e_{{n}}+e_{\pi(n)})\alpha\gamma(1+\delta-\beta(1-\delta)).
			%W_f(a+e_{\pi(n)})+\alpha W_f(a+e_{{n}})]
			%Pante: I commented this out, as it should not be there
		\end{align*}
		Recalling the expression of $W_h(a,b)$ from \eqref{Eq_WH_GF} as below, 
$$W_h(a,b)=G_{0,0}(b)\widetilde{F_{0,0}}(a)+G_{0,1}(b)\widetilde{F_{0,1}}(a)+G_{1,0}(b)\widetilde{F_{1,0}}(a)+G_{1,1}(b) \widetilde{F_{1,1}}(a),$$ 
		where $\widetilde{F_{i,j}}(a)=W_f(a + ie_{\pi(n)} + je_{n})$ for $i,j\in \F_2$,
		we see that 
		\[
		\begin{array}{ll}
		G_{0,0}(b)=(1+\delta+\beta(1-\delta)), & G_{1,0}(b)=\gamma(1-\delta+\beta(1+\delta)), \\
		G_{0,1}(b)=\alpha(1-\delta-\beta(1+\delta)),
		&
		G_{1,1}(b)=\alpha\gamma(1+\delta-\beta(1-\delta)).
		\end{array}
		\]
		Notice that 
		\[
		\begin{split}
		\prod_{i,j\in\{0,1\}}G_{i,j}(b) & = \alpha^2 \gamma^2 ((1+\delta)^2-\beta^2(1-\delta)^2) ((1-\delta)^2-\beta^2(1+\delta)^2) \\
		& = - ((1+\delta)^2- (1-\delta)^2)^2
		= -2^4.
		\end{split}
		\]
		Thus, it follows from Theorem~\ref{Th_F_G} $(ii)$ that $h(x,y)$ is bent if and only if 
        $f(x)$ satisfies the WHC
		\begin{equation*}
			\prod_{i,j\in\{0,1\}}\widetilde{F_{i,j}}(a)=2^{2n}.
		\end{equation*}
		Similarly as $\rho=\rho_1$, the functions $g_{i,j}(\bar{y})$ for $i,\,j\in \F_2$ are all bent.
		Hence for $\rho\in \{\rho_2,\rho_5,\rho_7,\rho_{10},\rho_{12}\}$, $\pi \rho^R \sim I_{n+4}$ if and only if $f(x)$ satisfies the WHC.
	\end{proof}
	%Applying the reverse operation, we have the following consequence.
	%\begin{corollary}\label{cor2}
	%    Suppose $\pi\in S_n$ is compatible with $I_n$ and the bent function $f = Q_{I_n} + Q_{\pi}$ satisfies either WHC or NWHC, then 
	%    the permutations $\rho^L\widetilde{\pi}$ for $\rho\in \{\rho_{11},\rho_5,\rho_2, \rho_{12},\rho_{9},\rho_{7}\}$ are compatible with $I_{n+4}$.
	%\end{corollary}
	
	\medskip
	
	For even $n$, Theorem \ref{Th_F_G}, Corollary \ref{101024th1} and Corollary \ref{121024th2} characterize the required properties of $\pi \in S_n$ when it is extended with $\rho_i\in IS_4$ on the right side. For odd $n$, we have Theorem \ref{Th_F_G_odd} and Corollary \ref{101024th1}. 
	Below we shall investigate  how one can 
	obtain permutations in $S_{4m}$ for $m\geq 2$ that are compatible with $I_{4m}$, when $\pi$ is picked from $IS_4$.
	
	\smallskip 
	
	We start with the case of $m=2$. For the 12 permutations in $IS_4$, 
	by a routine calculation, the sets of $\rho_i$ for $1\leq i\leq 12$ such that $f(x)=Q_{I_4}(x)+Q_{\rho_i}(x)$ satisfies \cl the WHC are given as follows,
	\[
	\begin{array}{rcl}
	S_{WHC}&=&\{\rho_1,\rho_3,\rho_8,\rho_9,\rho_{10},\rho_{12}\}.
	% S_2&=&\{\rho_2, \rho_5, \rho_4,\rho_6,\rho_9,\rho_{11}\}.
	\end{array}
	\] \cn 
	Below we shall recursively extend the permutations in $IS_4$ except for $\rho_2, \rho_5$ (which cannot be extended multiple times),
	thereby obtaining permutations in $IS_{4(k+1)}$ for $m=1,2,3,\dots$.
	Denote by
	\[
	\pi\rho^{R_{m}} = \pi\overbrace{\rho^R \cdots \rho^R}^{m}
	\] 
	the right extension on $\pi$ with $\rho$ by $m$ times. When $\pi = \rho$, we denote $\pi\rho^{R_{m}}$ as $\rho^{R_{m+1}}$.
	We have the following result.
\begin{theorem}\label{Th_2} 
Let $n=4m$ with an integer $m\geq 2$ and $IS_4 = \{\rho_1,\rho_2,\dots, \rho_{12}\}$.
Then, for any $\rho_i \in IS_4\setminus\{\rho_2,\rho_5\} $, the permutation $\rho_i^{R_m}$ is compatible with $ I_{4m}.$ 
\end{theorem}
	
\begin{proof}
By Corollary \ref{101024th1}, if $\rho_i\in \{\rho_3,\rho_4,\rho_6, \rho_8,\rho_9, \rho_{11}\}$, then $\rho_i^{R_{m-1}}\sim I_{4(m-1)}$ immediately implies $\rho_i^{R_{m}}\sim I_{4m}$. For the permutations $\rho_{10}, \rho_{12}$, since $\rho_{10}(4) = \rho_{12}(4)  = 4$, it follows from Theorem \ref{Th_F_G} $(i)$ that 
the permutations $\rho_{10}^{R_{m}}=\rho_{10}(\rho_{10}^{R_{m-1}})^{R}$ and $\rho_{12}^{R_{m}}=\rho_{12}(\rho_{12}^{R_{m-1}})^{R}$
are compatible with $I_{4m}$. 
		
We will prove the statement for the remaining permutations $\rho_1,\rho_7$,
by induction on $m$. We start with the discussion for $\rho_1$.

For $m=2$, it is clear that 
\[
\rho_1^{R_2}=\rho_1 \rho_1^R \sim I_8.
\]
Suppose $\rho_1^{R_{k}}$ is compatible with $I_{4k}$ for an integer $k$ with $2\leq k<m$. Then we need to show that 
$\rho_i^{R_{k+1}}$ is compatible with $I_{4(k+1)}$.
By Corollary \ref{121024th2}, the assumption $\rho_1^{R_{k}}\sim I_{4k}$
implies that the bent function 
		\ccb 
\[
f_{k}(x) = Q_{\rho_1^{R_{k}}}(x) + Q_{I_{4k}}(x)
\] 
satisfies the WHC on $(\pi(4k), 4k)$, where $\pi = \rho_1^{R_{k}}$. \cl 
This implies that for any $a\in \F_2^{4k}$, 
\begin{equation}\label{Eq_WHC-1}
			\frac{W_{f_k}(a+e_{4k}+e_{\pi(4k)})}{W_{f_k}(a)} = \frac{W_{f_k}(a+e_{4k})}{W_{f_k}(a+e_{\pi(4k)})}.   
		\end{equation}
		\cn
		Below we shall show that the bent function 
		\ccn
		\ccb  
		\[
		f_{k+1}(x) =  Q_{\rho_1^{R_{k+1}}}(x) + Q_{I_{4(k+1)}}(x)
		\] 
		\cl    satisfies the WHC on $(\pi'(4(k+1)), 4(k+1))$, where $\pi' = \pi\rho_1^R = \rho_1^{R_{k+1}}$. Note that for integers $j=1,2,\dots, 4k$, 
        $\pi'(j)=\pi(j)$.
        \cn
		
		Recall from the proof of Corollary~\ref{121024th2} that
		\allowdisplaybreaks
		\begin{align*}
			W_{f_{k+1}}(\overline{a})&=(1+\delta)\left(W_{f_k}(a)+\alpha\gamma W_{f_k}(a+e_{{4k}}+e_{\pi(4k)})\right.\\
			&\qquad\qquad\qquad\left.+\beta(\gamma W_{f_k}(a+e_{\pi(4k)})-\alpha W_{f_k}(a+e_{{4k}}))\right)\\
			&\quad+(1-\delta)\left(\beta(W_{f_k}(a)-\alpha\gamma W_{f_k}(a+e_{{4k}}+e_{\pi(4k)}))\right.\\
			&\left.\qquad\qquad\qquad+\gamma W_{f_k}(a+e_{\pi(4k)})+\alpha W_{f_k}(a+e_{{4k}})\right)\\
			&=(1+\delta)\left(W_{f_k}(a)(1+\epsilon\alpha\gamma) +\beta W_{f_k}(a+e_{\pi(4k)})(\gamma-\epsilon\alpha)\right)\\
			&\quad +(1-\delta)\left(\beta W_{f_k} (a)(1-\epsilon\alpha\gamma) + W_{f_k}(a+e_{\pi(4k)})(\gamma+\epsilon\alpha)\right),
		\end{align*}
		\ccb where $\overline{a}=(a_1,a_2,\hdots,a_{4k+4})$, $a=(a_1,a_2,\hdots,a_{4k})$,
		$\alpha=(-1)^{a_{4k+1}},\beta=(-1)^{a_{4k+2}},\gamma=(-1)^{a_{4k+3}}$, $\delta=(-1)^{a_{4k+4}}$, and $\epsilon$ denotes the division in \eqref{Eq_WHC-1} and takes the values $\pm 1$. \ccn Therefore,
		\begin{align*}
			&W_{f_{k+1}}(\overline{a}+e_{{4(k+1)}}+e_{\pi'(4(k+1))})\\
			&=(1-\delta)\left( W_{f_k}(a)(1-\epsilon\alpha\gamma) +\beta W_{f_k}(a+e_{\pi(4k)})(\gamma+\epsilon\alpha)\right)\\
			&\qquad +(1+\delta)\left(\beta W_{f_k}(a)(1+\epsilon\alpha\gamma) + W_{f_k}(a+e_{\pi(4k)})(\gamma-\epsilon\alpha)\right).
		\end{align*}
		From here, it is clear that $W_{f_{k+1}}(\overline{a}+e_{{4(k+1)}}+e_{\pi'(4(k+1))})=\beta W_{f_{k+1}}(\overline{a})$.
		In the same way,
		\begin{align*}
			W_{f_{k+1}}(\overline{a}+e_{\pi'(4(k+1))})
			&=(1+\delta)\left(W_{f_k}(a)(1-\epsilon\alpha\gamma) +\beta W_{f_k}(a+e_{\pi(4k)})(\gamma+\epsilon\alpha)\right)\\
			&+(1-\delta)\left(\beta W_{f_k}(a)(1+\epsilon\alpha\gamma) + W_{f_k}(a+e_{\pi(4k)})(\gamma-\epsilon\alpha)\right),
		\end{align*}
		while
		\begin{align*}
			W_{f_{k+1}}(\overline{a}+e_{{4(k+1)}}))
			&=(1-\delta)\left(W_{f_k}(a)(1+\epsilon\alpha\gamma) +\beta W_{f_k}(a+e_{\pi(4k)})(\gamma-\epsilon\alpha)\right)\\
			&+(1+\delta)\left(\beta W_{f_k}(a)(1-\epsilon\alpha\gamma) + W_{f_k}(a+e_{\pi(4k)})(\gamma+\epsilon\alpha)\right).
		\end{align*}
		It follows, therefore, that $W_{f_{k+1}}(\overline{a}+e_{\pi'(4(k+1))})=\beta W_{f_{k+1}}(\overline{a}+e_{{4(k+1)}})$.
		This implies that $f_{k+1}(x)$ fulfills the WHC on $(\pi'(4(k+1)),4(k+1))$. Therefore, $\rho_1^{R_{k+1}}$ is compatible with $I_{4(k+1)}$. 
		\cl 

        \smallskip 
        
        As for the permutation $\rho_7$, we have $\rho_7^{-1}= \rho_1$. According to \eqref{Eq_Ext_Inv_Rev},
		it follows that $(\rho_7\rho_7^R)^{-1} = (\rho_7^{-1})(\rho_7^{-1})^R = \rho_1 \rho_1^R$ and recursively, $(\rho_7^{R_{m}})^{-1} = \rho_1^{R_{m}}$.
		Thus it follows from Lemma~\ref{lem_properties} $(ii)$ that $\rho_7^{R_{m}} \sim I_{4m}$.
		\cn	
		This concludes the proof.
	\end{proof}

	%In order to construct a compatible set, we need to first choose a proper starting permutation $\pi$ that is compatible with the identity permutation, and then further derive mutually compatible permutations from those obtained by different extensions. In the following, we consider the starting permutations from $IS_4$.

	\medskip
	
	\ccn
	
	In the following, we present a way to extend a maximal compatible set in dimension 4 (which has size 6) to a size-6 set in any dimension $4m$ for $m>1$, by recursively adding a shift of itself.
	\begin{theorem}	\label{Th_Main_Const}
		Given any maximal compatible set $\Pi
		$ in dimension $4$, 
		the set $\{ \rho^{R_m}:~ \rho \in \Pi\}$  
		is a compatible set in dimension $4m$ for any $m\geq 2$. 
	\end{theorem}
	\begin{proof}
		Let $\pi,\sigma\in \Pi$. We know that 
		$\pi \sim \sigma$ and $\pi, \sigma \not\in \{\rho_2, \rho_5\}$.
		We shall show $\pi^{R_m}$ is compatible with $\sigma^{R_m}$, equivalently,
		$(\pi^{R_m})^{-1}\circ \sigma^{R_m}\sim I_{4m}$. 
		
		From $\pi \sim \sigma$, we know $\pi^{-1}\circ \sigma \sim I_4$ and $\sigma^{-1}\circ\pi\sim I_4$. Without loss of generality, we assume $\pi^{-1}\circ \sigma\sim I_4$. 
		According to \eqref{Eq_Ext_Inv_Rev} and \eqref{Eq_Ext_Rev}, we have 
		\allowdisplaybreaks
		\[
		\begin{split}
		(\pi^{R_m})^{-1}\circ \sigma^{R_m}
		&= (\pi ^{R_{m-1}}\pi^R )^{-1}\circ\sigma^{R_m}  
		\\
		&=((\pi ^{R_{m-1}})^{-1}(\pi^{-1})^{R})\circ \sigma^{R_m} 
		\\
		&=((\pi ^{R_{m-1}})^{-1}(\pi^{-1})^{R})\circ  (\sigma^{R_{m-1}}\sigma^R) 
		\\
		&= ((\pi^{R_{m-1}})^{-1}\circ \sigma ^{R_{m-1}}) ( \pi^{-1}\circ \sigma)^R \\
		& = ((\pi^{R_{m-2}})^{-1}\circ \sigma ^{R_{m-2}}) ( \pi^{-1}\circ \sigma)^{R_2} \\
		& = \ldots \\
		&= ( \pi^{-1}\circ \sigma)^{R_m},
		\end{split}
		\]
		where the second equality follows from~\eqref{Eq_Ext_Inv_Rev} and the fourth equation follows from~\eqref{Eq_Ext_Rev}.
		Notice that $\pi^{-1}\circ \sigma$ is compatible with $I_4$, 
		$\pi^{-1}\circ \sigma$ is neither $\rho_2$ nor $\rho_5$ (as shown in Table \ref{tab_invcomp}). It follows from Theorem \ref{Th_2}  that
		$(\pi^{-1}\circ \sigma)^{R_m}$ is compatible with $I_{4m}$, 
		implying
		$(\pi^{R_m})^{-1}\circ \sigma^{R_m}\sim I_{4m}$, and then $\pi^{R_m} \sim \sigma^{R_m}$.
		Since $\pi, \, \sigma$ are freely chosen from $\Pi$, one can easily see that the set $\{\rho^{R_m}\,:\, \rho \in \Pi\}$ is a compatible set. 
	\end{proof}
	According to Table \ref{tab_CS_4}, Theorem \ref{Th_Main_Const} yields $32$ compatible sets each consisting of $6$ \cli permutations \cn in $S_{4m}$ for all integers $m\geq 2$.
	
	%\subsection{Other extensions of sets of cardinality 6} 
    
	\begin{remark} 
		For dimension $4m$, it is possible to consider other types of extension. For other extensions than repetition, to ensure newly extended permutations are still compatible,
		  one needs to check the required condition as in Lemma~\textup{\ref{lem_mutual_compatible}}. Specifically, 
		when $\pi\sim \sigma$ and they are extendable by $\rho_i, \rho_j$, respectively,
		if both $\rho_j^{-1}\circ\rho_i$ and $\rho_i^{-1}\circ\rho_j$  require 
		the WHC, then $\pi^{-1}\circ\sigma$ or $\sigma^{-1}\circ\pi$ must satisfy the corresponding WHC. Depending on the choice of permutations, additional checks may be required in the recursive extension. 
	\end{remark}
\cli  In Appendix A, we present several examples of compatible sets comprising six permutations in dimension $4m$ for $m=2,3,4$.  \cn

	\section{Conclusion}\label{Sec:Con} 
	
	In this paper, we propose a recursive construction of a codebook for uplink grant-free NOMA using GDJ sequences. 
	The contribution is twofold: we establish the necessary and sufficient condition for a permutation of type \( \pi\rho^R \) to be compatible with \( I_{n+m} \) when \( \pi\) is compatible with $I_n$, and \( \rho \) is compatible with $I_m$, respectively; and we recursively extend compatible sets of \( 4 \)-dimensional permutations to compatible sets of \( 4m \)-dimensional permutations. As a result, any compatible set in dimension 4 can be extended to a compatible set of the same size in dimension \( 4m \). The proposed approach allows for constructing many NOMA codebooks of $6N$ GDJ sequences of length $N=2^{4m}$ and the lowest possible coherence $1/\sqrt{N}$ for integers $m\geq 1$.

	\ccn

	\ccb
	\vskip.5cm
	\noindent
%	\section*{Acknowledgment}
%	The authors would like to thank the editors for efficiently handling our paper and the reviewers for their careful reading, beneficial comments, and constructive suggestions. We thank Prof. Dibyendu Roy for his valuable input on our SageMath implementation.  The fourth-named author worked on this problem during a visit to the Selmer Center at the University of Bergen. He thanks the center for the hospitality, support, and excellent working conditions. 

	\ccn

	
	%\bibliography{nomacodebookbfa24}% common bib file
	

\appendix

\cli

\section{Explicit examples of compatible sets for generating NOMA codebooks}
\label{sec:appendix_examples}

The examples presented in this appendix were generated and verified using a computational script in SageMath. The script implements the exhaustive search for the base case of $n=4$ and the recursive self-extension method described in Theorem \ref{Th_Main_Const} as well as mixed-extension methods.

\subsection{Base case: maximal compatible sets for $n=4$}

The recursive construction starts with a maximal compatible set for $n=4$, containing $L=6$ permutations by exhaustion. Table \ref{tab:n4_sets} lists two such sets, $\Pi_1$ and $\Pi_2$. The set $\Pi_1$ corresponds to the first set listed in Table \ref{tab_CS_4}, while $\Pi_2$ and $\Pi_3$ correspond to the fourth set and 27th set, respectively. These sets form the $m=1$ base case, yielding codebooks of $K=96$ sequences of length $N=16$ with optimally low coherence $\mu(\bPhi) = 1/4$.

\begin{table}[h!]
\centering
\cli
\caption{\cli Example maximal compatible sets for $n=4$ (base case).}
\label{tab:n4_sets}
\begin{tabular}{ll}
\hline
\textbf{Set} & \textbf{Permutations (1-indexed)} \\
\hline
\multirow{2}{*}{$\Pi_1$} & $\{[1, 2, 3, 4], [3, 2, 4, 1], [3, 4, 1, 2],$\\
& $[1, 3, 4, 2], [4, 2, 1, 3], [4, 1, 3, 2]\}$ \\
\hline
\multirow{2}{*}{$\Pi_2$} & $\{[1, 2, 3, 4], [2, 3, 1, 4], [3, 4, 1, 2],$\\
& $[2, 4, 3, 1], [4, 2, 1, 3], [1, 4, 2, 3]\}$ \\
\hline
\multirow{2}{*}{$\Pi_3$} & 
$\{
[ 1, 2, 3, 4 ],
[ 1, 3, 4, 2 ],
[ 1, 4, 2, 3 ], $\\ & $
        [ 2, 3, 1, 4 ],
        [ 3, 4, 1, 2 ],
        [ 4, 2, 1, 3 ]
\}$ \\
\hline
\end{tabular}
\end{table}

\subsection{Recursive self-extension examples: $n=4m$}

For dimension $n=4m$, Theorem \ref{Th_Main_Const} guarantees that the recursive self-extension of any base compatible set $\Pi$ maintains compatibility. This yields a set of $L=6$ permutations $\Pi^{R_m} = \{ \rho^{R_m} \mid \rho \in \Pi \}$ in $S_{4m}$.
Below we illustrate this with $\Pi_1$ for $m=2,3,4$ as in Tables \ref{tab:n8_self_extension}-\ref{tab:n16_self_extension}, respectively.
Correspondingly, we obtain a $2^{4m}\times 6\cdot 2^{4m}$ NOMA codebook of sequences with length $2^{4m}$ and PAPR upper bouned by $2$ and optimally low coherrence $1/2^{2m}$.

\begin{table}[h!]
\centering
\cli
\caption{\cli Self-recursive extension on $\Pi_1$ for $m=2$.}
\label{tab:n8_self_extension}
\begin{tabular}{ll}
\hline
\textbf{Base $\pi_k \in \Pi_1$} & \textbf{Extended $\pi_k^{R_2} = \pi_k \cdot \pi_k^R$} \\
\hline
$[1, 2, 3, 4]$ & $[1, 2, 3, 4, 5, 6, 7, 8]$ \\
$[3, 2, 4, 1]$ & $[3, 2, 4, 1, 7, 6, 8, 5]$ \\
$[3, 4, 1, 2]$ & $[3, 4, 1, 2, 7, 8, 5, 6]$ \\
$[1, 3, 4, 2]$ & $[1, 3, 4, 2, 5, 7, 8, 6]$ \\
$[4, 2, 1, 3]$ & $[4, 2, 1, 3, 8, 6, 5, 7]$ \\
$[4, 1, 3, 2]$ & $[4, 1, 3, 2, 8, 5, 7, 6]$ \\
\hline
\end{tabular}
\end{table}

\begin{table}[h!]
\centering
\cli
\caption{\cli Self-recursive extension on $\Pi_1$ for $m=3$.}
\label{tab:n12_self_extension}
\begin{tabular}{ll}
\hline
\textbf{Base $\pi_k \in \Pi_1$} & \textbf{Extended $\pi_k^{R_3} = \pi_k^{R_2} \cdot \pi_k^R$} \\
\hline
$[1, 2, 3, 4]$ & $[1, 2, 3, 4, 5, 6, 7, 8, 9, 10, 11, 12]$ \\
$[3, 2, 4, 1]$ & $[3, 2, 4, 1, 7, 6, 8, 5, 11, 10, 12, 9]$ \\
$[3, 4, 1, 2]$ & $[3, 4, 1, 2, 7, 8, 5, 6, 11, 12, 9, 10]$ \\
$[1, 3, 4, 2]$ & $[1, 3, 4, 2, 5, 7, 8, 6, 9, 11, 12, 10]$ \\
$[4, 2, 1, 3]$ & $[4, 2, 1, 3, 8, 6, 5, 7, 12, 10, 9, 11]$ \\
$[4, 1, 3, 2]$ & $[4, 1, 3, 2, 8, 5, 7, 6, 12, 9, 11, 10]$ \\
\hline
\end{tabular}
\end{table}

\begin{table}[h!]
\centering
\cli
\caption{\cli Self-recursive extension on $\Pi_1$ for $m=4$.}
\label{tab:n16_self_extension}
\begin{tabular}{ll}
\hline
\textbf{Base $\pi_k \in \Pi_1$} & \textbf{Extended $\pi_k^{R_4}$} \\
\hline
$[1, 2, 3, 4]$ & $[1, 2, 3, 4, 5, 6, 7, 8, 9, 10, 11, 12, 13, 14, 15, 16]$ \\
$[3, 2, 4, 1]$ & $[3, 2, 4, 1, 7, 6, 8, 5, 11, 10, 12, 9, 15, 14, 16, 13]$ \\
$[3, 4, 1, 2]$ & $[3, 4, 1, 2, 7, 8, 5, 6, 11, 12, 9, 10, 15, 16, 13, 14]$ \\
$[1, 3, 4, 2]$ & $[1, 3, 4, 2, 5, 7, 8, 6, 9, 11, 12, 10, 13, 15, 16, 14]$ \\
$[4, 2, 1, 3]$ & $[4, 2, 1, 3, 8, 6, 5, 7, 12, 10, 9, 11, 16, 14, 13, 15]$ \\
$[4, 1, 3, 2]$ & $[4, 1, 3, 2, 8, 5, 7, 6, 12, 9, 11, 10, 16, 13, 15, 14]$ \\
\hline
\end{tabular}
\end{table}

\subsection{Recursive mixed-extension examples}
Exhaustive search reveals that one can also obtain compatible sets of size $6$
by \textit{mixed extension}: extending each permutation in one compatible set on the right with permutations from another compatible set. Suppose $\Pi_A$ is a compatible set in dimension $4m_1$ and 
$\Pi_B$ is another compatible set in dimension $4m_2$. We start with a candidate set $\Pi_A\Pi_B^R = \{\pi\rho^R: \pi\in\Pi_A, \rho\in \Pi_B\}$, which contains in total $\vert \Pi_A \vert \cdot \vert \Pi_B \vert$ permutations in dimension $4(m_1+m_2)$.
This set is then fed into a function in our implementation, which outputs all maximal compatible sets for the given input. As a result, we obtain many examples of compatible sets of size $6$ in higher dimensions. 

Taking $(\Pi_A, \Pi_B) = (\Pi_1, \Pi_3)$ as in Table~\ref{tab:n4_sets}, one can obtain four compatible sets in dimension $8$ of size $6$. Table~\ref{tab:n8_mixed_extension} gives an example for $m=2$.

% \begin{table}[h!]
% \centering
% \cli
% \caption{\cli Base sets for mixed extension at $n=8$.}
% \label{tab:n8_mixed_bases}
% \begin{tabular}{ll}
% \hline
% \textbf{Set} & \textbf{Permutations (1-indexed)} \\
% \hline
% \multirow{2}{*}{$\Pi_A$} & $\{[1, 2, 3, 4], [1, 3, 4, 2], [1, 4, 2, 3],$\\
% & $[2, 1, 4, 3], [2, 3, 1, 4], [3, 1, 2, 4]\}$ \\
% \hline
% \multirow{2}{*}{$\Pi_B$} & $\{[1, 2, 3, 4], [1, 3, 4, 2], [1, 4, 2, 3],$\\
% & $[3, 1, 2, 4], [3, 4, 1, 2], [4, 1, 3, 2]\}$ \\
% \hline
% \end{tabular}
% \end{table}

% The mixed extension $\Pi_A\Pi_B^R = \{\pi\rho^R: \pi\in\Pi_A, \rho\in \Pi_B\}$ yields a compatible set for $n=8$ (Table~\ref{tab:n8_mixed_extension}), demonstrating that alternatives to self-recursive extension exist. However, not all pairings of compatible base sets produce compatible mixed extensions; for instance, $\Pi_1 \times \Pi_2$ from Table~\ref{tab:n4_sets} fails the compatibility criterion.

\begin{table}[h!]
\centering
\cli
\caption{\cli An example of mixed extension for $m=2$.}
\label{tab:n8_mixed_extension}
\begin{tabular}{lll}
\hline
\textbf{$\pi \in \Pi_1$} & \textbf{$\rho \in \Pi_3$} & \textbf{Extended $\pi\rho^R$} \\
\hline
$[1, 2, 3, 4]$ & $[1, 2, 3, 4]$ & $[ 1, 2, 3, 4, 5, 6, 7, 8 ]$ \\
$[2, 4, 3, 1]$ & $[1, 4, 2, 3]$ & $[ 2, 4, 3, 1, 5, 8, 6, 7 ]$ \\
$[3, 2, 4, 1]$ & $[3, 4, 1, 2]$ & $[ 3, 2, 4, 1, 7, 8, 5, 6 ]$ \\
$[3, 4, 1, 2]$ & $[2, 3, 1, 4]$ & $[ 3, 4, 1, 2, 6, 7, 5, 8 ]$ \\
$[4, 1, 3, 2]$ & $[1, 3, 4, 2]$ & $[ 4, 1, 3, 2, 5, 7, 8, 6 ]$ \\
$[4, 2, 1, 3]$ & $[4, 2, 1, 3]$ & $[ 4, 2, 1, 3, 8, 6, 5, 7 ]$ \\
\hline
\end{tabular}
\end{table}

Take $\Pi_A$ as the compatible set in Table~\ref{tab:n8_mixed_extension} and $\Pi_B$ as $\Pi_3$ again. Similarly, one obtains several compatible sets of size $6$.  
Table~\ref{tab:n12_mixed_extension} gives an example for $m=3$.

\begin{table}[h!]
\centering
\cli
\caption{\cli An example of mixed extension for $m=3$.}
\label{tab:n12_mixed_extension}
\begin{tabular}{lll}
\hline
\textbf{$\pi \in \Pi_A$} & \textbf{$\rho \in \Pi_3$} & \textbf{Extended $\pi\rho^R$} \\
\hline
$[ 1, 2, 3, 4, 5, 6, 7, 8 ]$  & $[1, 2, 3, 4]$ & $[ 1, 2, 3, 4, 5, 6, 7, 8, 9, 10, 11, 12 ]$ \\
$[ 2, 4, 3, 1, 5, 8, 6, 7 ]$ & $[1, 4, 2, 3]$& $[ 2, 4, 3, 1, 5, 8, 6, 7, 9, 12, 10, 11 ]$\\
$[ 3, 2, 4, 1, 7, 8, 5, 6 ]$ & $[3,4,1,2]$ & $[ 3, 2, 4, 1, 7, 8, 5, 6, 11, 12, 9, 10 ]$\\
$[ 3, 4, 1, 2, 6, 7, 5, 8 ]$ & $[2,3,1,4]$ & $[ 3, 4, 1, 2, 6, 7, 5, 8, 10, 11, 9, 12 ]$\\
$[ 4, 1, 3, 2, 5, 7, 8, 6 ]$ & $[1,3,4,2]$& $[ 4, 1, 3, 2, 5, 7, 8, 6, 9, 11, 12, 10 ]$\\
$[ 4, 2, 1, 3, 8, 6, 5, 7 ]$ & $[4,2,1,3]$& $[ 4, 2, 1, 3, 8, 6, 5, 7, 12, 10, 9, 11 ]$\\
\hline
\end{tabular}
\end{table}

Experimental results demonstrate that the mix-extension method can generate many compatible sets of size $6$ in a dimension of $4m$ for $m\geq 2$. It is worth noting that one is not necessarily restricted to extending a compatible set by another compatible set. Instead, one may extend a set of permutations with another set of permutations. Starting from $n=4$, among all permutations in $IS_4$, one can get a candidate set of 169 permutations in dimension~$8$, which leads to $1936$ compatible sets of size $6$ in dimension $8$.
This type of extension can continue with larger $m$.

The theoretical analysis of the recursive mixed-extension approach appears complicated by using the technique in this paper. We shall further develop new methods for those cases in our future work.

\subsection{Summary of recursive construction parameters}

The recursive construction generates NOMA codebooks of size $K=6N$ for all dimensions $n=4m$, where each sequence has length $N=2^n$ and PAPR upper bounded by $2$.

\begin{table}[h!]
\centering
\cli
\caption{\cli Properties of NOMA codebooks from the recursive construction.}
\label{tab:recursive_summary}
\begin{tabular}{cccccc}
\hline
$m$ & Dim. $n$ & Len. $N$ & \#Sequences $K$ & Coherence $\mu(\bPhi)$ & Overloading Factor $L$ \\
\hline
1 & 4 & 16 & 96 & $1/4$ & 6 \\
2 & 8 & 256 & 1,536 & $1/16$ & 6 \\
3 & 12 & 4,096 & 24,576 & $1/64$ & 6 \\
4 & 16 & 65,536 & 393,216 & $1/256$ & 6 \\
\hline
\end{tabular}
\end{table}

All codebooks achieve the optimal coherence $1/\sqrt{N}$ as shown in \eqref{Eq_coh_lowerbounds}, and the construction maintains a constant overloading factor $L=6$ for all dimensions $n=4m$ with $m \geq 1$.

\cn

\end{document}